\documentclass[11pt,reqno]{amsart}

\usepackage[margin=1in]{geometry}
\usepackage{amsmath,amssymb,amsthm}
\usepackage{mathrsfs}
\usepackage{enumitem}
\usepackage[dvipsnames]{xcolor}
\usepackage[colorlinks=true,linkcolor=blue,citecolor=blue,urlcolor=blue]{hyperref}

\numberwithin{equation}{section}
\newtheorem{theorem}{Theorem}[section]
\newtheorem{proposition}[theorem]{Proposition}
\newtheorem{lemma}[theorem]{Lemma}
\newtheorem{corollary}[theorem]{Corollary}
\theoremstyle{definition}

\newtheorem{assumption}[theorem]{Assumption}
\theoremstyle{remark}
\newtheorem{remark}[theorem]{Remark}
\newcommand{\Lh}{\hyperref[ass:L]{\textup{(L)}}}
\newcommand{\Gh}{\hyperref[ass:G]{\textup{(G)}}}

\newcommand{\Ric}{\operatorname{Ric}}
\newcommand{\Def}{\operatorname{Def}}
\newcommand{\dv}{\operatorname{div}}
\newcommand{\tr}{\operatorname{tr}}
\newcommand{\supp}{\operatorname{supp}}
\newcommand{\vol}{\operatorname{vol}}
\newcommand{\dist}{\operatorname{dist}}
\newcommand{\Hess}{\nabla^2}
\newcommand{\DHo}{\Delta_{H}}
\newcommand{\DB}{\Delta_{B}}
\newcommand{\cH}{\mathscr{H}}
\newcommand{\cA}{\mathscr{A}}
\newcommand{\cN}{\mathscr{N}}
\newcommand{\cO}{\mathscr{O}}
\newcommand{\cB}{\mathscr{B}}

\newcommand{\Lam}{\mathbb{L}}
\newcommand{\PP}{\mathbb{P}}
\newcommand{\R}{\mathbb{R}}
\newcommand{\C}{\mathbb{C}}
\newcommand{\dd}{\mathrm{d}}
\newcommand{\inte}{\operatorname{int}}
\newcommand{\ww}[1]{\langle\!\langle #1\rangle\!\rangle}

\begin{document}

\title[Hodge--Carleman estimates and attractors for heterogeneous elastic waves]{Hodge--Carleman estimates and global attractors for heterogeneous elastic waves with localized damping on Riemannian manifolds}

\author{Paulo N. Seminario-Huertas}
\thanks{Departamento de Matem\'atica Aplicada a las Tecnolog\'ias de la Informaci\'on y las Comunicaciones, Universidad Polit\'ecnica de Madrid, Campus de Montegancedo, 28660 Boadilla del Monte, Madrid, Spain. \textit{Email address}: \texttt{paulo.seminario.huertas@upm.es}}


\begin{abstract}
	We study the long-time dynamics of the semilinear Lam\'e system of isotropic elasticity on a compact Riemannian manifold with boundary, for a heterogeneous medium whose density and Lam\'e moduli vary in space, under a frictional dissipation which acts only on an open set $\omega$. The elastic operator is the divergence of the Hooke stress tensor, and the starting point is a structural identity which shows that the dilatation and the vorticity of the displacement satisfy wave equations with the pressure and the shear speeds, respectively, coupled to the displacement only through first-order terms. A single weight function, free of critical points and strictly convex with respect to the two ray metrics of the medium, then serves for the displacement, its dilatation and its vorticity alike, and this Hodge structure yields Carleman estimates for the Lam\'e operator with variable coefficients at two energy levels, provided that $\omega$ contains a neighbourhood of the boundary and that such a weight exists on the undamped region. The hypothesis is fulfilled on every compact manifold with boundary by damping regions of arbitrarily small volume. From these estimates we derive unique continuation for elastic waves with bounded, merely measurable, time-dependent potentials and an internal observability inequality, which in turn yield, for subcritical restoring forces in dimension three, a compact global attractor which coincides with the unstable manifold of the equilibria, has finite fractal dimension and is bounded in $(H^2\cap H^1_0)\times H^1_0$, together with a generalized exponential attractor.
\end{abstract}

\maketitle

%=====================================================================
\section{Introduction}\label{sec:intro}
%=====================================================================

This paper is concerned with the long-time dynamics of elastic waves in a curved and heterogeneous elastic body when the dissipation acts only on a part of it. Let $(M,g)$ be a compact connected Riemannian manifold of dimension $n\ge 2$ with smooth boundary $\partial M$. We consider the semilinear Lam\'e system
\begin{equation}\label{eq:main}
\begin{cases}
\varrho(x)\partial_t^2u-\Lam u+a(x)\partial_tu+f(x,u)=h(x) & \text{in } (0,\infty)\times M,\\
u=0 & \text{on } (0,\infty)\times\partial M,\\
u(0)=u_0,\quad \partial_tu(0)=u_1 & \text{in } M,
\end{cases}
\end{equation}
where the unknown $u(t,x)\in T_xM$ is the displacement field, $\varrho>0$ is the density, $a\ge0$ is a damping coefficient which is positive only on an open set $\omega\subset M$, $f$ is a nonlinear restoring force and $h$ is an external load. The elastic operator is
\begin{equation}\label{eq:Lame-def}
\Lam u=\dv\sigma(u),\qquad \sigma(u)=\lambda(\dv u)\,g+2\mu\Def u,\qquad \Def u=\tfrac12\mathcal{L}_ug,
\end{equation}
that is, the divergence of the Hooke stress tensor of an isotropic material whose Lam\'e moduli $\lambda$ and $\mu$ are smooth functions on $M$. The density and the Lam\'e moduli are subject only to the positivity conditions of Assumption~\ref{ass:L}.

When the Euclidean Lam\'e operator is transplanted to a curved background, one meets the ambiguity that different Laplacians acting on vector fields (the Bochner Laplacian, the Hodge Laplacian, or the operator $\Def^*\Def$) coincide in flat space and differ by curvature terms on a manifold. The same issue arises for the viscous term of the Navier--Stokes equations, as discussed by Chan, Czubak and Disconzi \cite{CCD17}. In this paper we adopt the operator \eqref{eq:Lame-def}, which is the one derived from the stored energy of linearized elasticity on $(M,g)$, see Marsden and Hughes \cite{MH83} and Grubi\'c, LeFloch and Mardare \cite{GLM14}, as well as Ciarlet \cite{Ciarlet05} for the same operator written in curvilinear coordinates of a Euclidean body, and which also appears in the inverse problem of Tan and Liu \cite{TL22}. In Proposition~\ref{prop:three-forms} below we show that
\begin{equation}\label{eq:intro-forms}
\Lam u=-\mu\DHo u-(\lambda+\mu)\,\dd\delta u+2\mu\Ric(u)+\cB u,
\end{equation}
where $\DHo=\dd\delta+\delta\dd$ is the Hodge Laplacian, $\Ric$ is the Ricci tensor and $\cB$ is a first-order operator which vanishes when $\lambda$ and $\mu$ are constant.

Global attractors for semilinear Lam\'e systems were obtained in \cite{BSMS23}, for Euclidean domains, frictional damping acting on the whole body and nonlinear forces of critical growth. Related results concern weakly damped Lam\'e systems with delay in \cite{MMS21}, degenerate memory and time-varying delay in \cite{HYY24}, non-autonomous thermoelastic Lam\'e systems in \cite{QW24}, supercritical sources with nonlinear damping in \cite{YZSZ26}, and variable Lam\'e coefficients in \cite{BMS26}. In all these works the domain is Euclidean and the dissipation acts on the whole body. On the other hand, the stabilization of elastic waves by a dissipation of restricted support has a long history. Boundary stabilization was established by Lagnese \cite{Lag83}, by Alabau and Komornik \cite{AK99} and by Horn \cite{Horn98b}, the latter on the basis of the sharp trace regularity obtained in \cite{Horn98}, and the boundary controllability was studied by Belishev and Lasiecka \cite{BL02} and by Dehman and Raymond \cite{DR15}. Dissipations supported in a neighbourhood of a part of the boundary were considered by Guesmia \cite{Gue00} and by Astaburuaga and Char\~ao \cite{AC02}, and localized interior dissipations under microlocal conditions by Bellassoued \cite{Bel03} and by Bchatnia and Daoulatli \cite{BD13}, the latter with nonlinear damping and an external force. Internal observability and stabilization of inhomogeneous and anisotropic elastic waves in Euclidean domains, with $\omega$ a neighbourhood of the whole boundary, were obtained by Meng, Ning and Yang \cite{MNY21} by multipliers of Morawetz type under a one-sided radial condition on the coefficients, and the propagation of microlocal energy for the Lam\'e system, including the conversion of pressure and shear modes at the boundary, was described by Burq and Lebeau \cite{BLe01}. For the scalar wave equation with localized damping, global attractors were obtained by Feireisl and Zuazua \cite{FZ93} and by Chueshov, Lasiecka and Toundykov \cite{CLT08} in Euclidean domains, by Joly and Laurent \cite{JL13} on compact manifolds under the geometric control condition of Bardos, Lebeau and Rauch \cite{BLR92}, and in \cite{CMMS26} on Riemannian manifolds under an obstacle condition. To the best of our knowledge, the long-time dynamics of semilinear Lam\'e systems with localized damping has not been studied before, either in Euclidean domains or on Riemannian manifolds.

The main obstruction is unique continuation. In the theory of attractors for damped waves one has to show that a trajectory whose energy is constant must be an equilibrium. When the damping acts only on $\omega$, this amounts to the following property. If $z$ solves
\begin{equation}\label{eq:intro-ucp}
\varrho\,\partial_t^2z-\Lam z+V(t,x)z=0,\qquad z=0\ \text{in } (0,T)\times\omega,
\end{equation}
then $z\equiv0$, where $V$ is a potential generated by the trajectory itself. Such a potential is in general neither smooth nor analytic in time. For the Lam\'e system, unique continuation is well understood in the stationary case, see Dehman and Robbiano \cite{DRob93}, Ang, Ikehata, Trong and Yamamoto \cite{AITY98}, Weck \cite{Wec01}, Alessandrini and Morassi \cite{AM01}, Escauriaza \cite{Esc06}, Lin, Nakamura and Wang \cite{LNW10} and Yu \cite{Yu11}. In the evolutionary case, Carleman estimates and uniqueness results were obtained by Eller, Isakov, Nakamura and Tataru \cite{EINT02}, Bellassoued \cite{Bel01}, Imanuvilov, Isakov and Yamamoto \cite{IIY03} and Imanuvilov and Yamamoto \cite{IY05,IY05b}, see also the monograph of Bellassoued and Yamamoto \cite{BY17}. All these results are Euclidean and rely on the fact that $\dv u$ and $\operatorname{curl}u$ satisfy wave equations coupled only through lower-order terms. On the other hand, the sharp uniqueness theorems of Tataru \cite{Tat95}, Robbiano and Zuily \cite{RZ98}, H\"ormander \cite{Hor97} and Laurent and L\'eautaud \cite{LL19} require coefficients which are analytic in time, and this is the reason why analyticity of the nonlinearity is assumed in \cite{JL13}.

Our approach is based on the Hodge form \eqref{eq:intro-forms}. Let $c_S=(\mu/\varrho)^{1/2}$ and $c_P=((\lambda+2\mu)/\varrho)^{1/2}$ be the shear and pressure speeds, and let $\PP=\partial_t^2-\varrho^{-1}\Lam$. In Proposition~\ref{prop:hodge} we show that
\begin{equation}\label{eq:intro-P}
\delta\,\PP u=\bigl(\partial_t^2+c_P^2\DHo\bigr)(\delta u)+\mathcal R_0(\delta u,\dd u,u),\qquad \dd\,\PP u=\bigl(\partial_t^2+c_S^2\DHo\bigr)(\dd u)+\mathcal R_2(\delta u,\dd u,u),
\end{equation}
where $\mathcal R_0$ and $\mathcal R_2$ are first-order differential operators with smooth coefficients. Thus the dilatation $\delta u=-\dv u$ propagates with the pressure speed and the vorticity $\dd u$ with the shear speed, while $u$ itself is governed by the Hodge wave operator with the shear speed, all the couplings being of first order. The two families of waves travel along the geodesics of the ray metrics $g_S=c_S^{-2}g$ and $g_P=c_P^{-2}g$, and a single weight function serves for both speeds as soon as it is strictly convex with respect to $g_S$ and $g_P$. In Euclidean domains, the pseudoconvexity of a single weight with respect to both wave operators, with the shear and the pressure speeds, is the hypothesis of \cite[Theorem~3.1]{IIY03}; its formulation in terms of two metrics appears in Isakov, Wang and Yamamoto \cite{IWY07} and in Isakov and Kim \cite{IK08,IK09}, for the Lam\'e system with residual stress, and it is the hypothesis of the Carleman estimate with source in $H^1$ of \cite[Lemma~7.4]{BY17}, which concerns compactly supported fields in a domain of $\R^3$, with moduli of class $C^2$, unit density and residual stress. An explicit radial form of the same condition is used in \cite{BIY08}. In one space dimension, where the speeds are constant and a common quadratic weight is admissible for all the metrics, the same device appears in \cite{MMS25} for Bresse systems. When the Poisson ratio is constant, and in particular for a Poisson solid, the two convexity requirements coincide. Combining \eqref{eq:intro-P} with the Riemannian Carleman estimates of \cite{BY17} and with the Carleman estimates in negative Sobolev norms of \cite{IIY03}, we derive Carleman estimates for the Lam\'e operator on $(M,g)$. We call them Hodge--Carleman estimates. The name is a label which records how they are obtained, through the Hodge form \eqref{eq:intro-forms} and the commutation of $\dd$ and $\delta$ with $\DHo$, and not a claim that Hodge theory enters beyond this point; in particular Theorem~\ref{thm:HC2} is the counterpart, on a Riemannian manifold and with variable density, of \cite[Lemma~7.4]{BY17}.

The main features of the paper are the following.
\begin{enumerate}[label=(\alph*)]
\item Two Hodge--Carleman estimates for the Lam\'e operator with variable coefficients on Riemannian manifolds (Theorems~\ref{thm:HC2} and~\ref{thm:HC0}). They hold under the geometric hypothesis \Gh, stated in Assumption~\ref{ass:G}, which requires that $\omega$ contains a neighbourhood of $\partial M$ and that the undamped region $M\setminus\omega$ carries a function without critical points which is strictly convex with respect to $g_S$ and $g_P$. The first estimate controls the solution up to second order in terms of the source in $H^1$, while the second one controls the solution, its dilatation and its vorticity in $L^2$ in terms of the source in $L^2$.
\item On every compact Riemannian manifold with boundary and for every $\varepsilon>0$ there exist damping regions of volume smaller than $\varepsilon$ satisfying \Gh{} (Proposition~\ref{prop:small}). Moreover, \Gh{} implies that every pressure ray and every shear ray meets $\omega$ in a uniform time (Proposition~\ref{prop:rays}).
\item Unique continuation for elastic waves on $(M,g)$ with bounded, merely measurable and time-dependent potentials (Theorem~\ref{thm:ucp}). As for every uniqueness theorem derived from a Carleman estimate, and in particular for the Euclidean results of \cite{IIY03}, no regularity of the potential beyond boundedness is needed; we state the result because it is the form in which uniqueness is needed in (e), and because it is not covered by the uniqueness theorems which require analyticity in time.
\item An internal observability inequality for the linear Lam\'e system (Theorem~\ref{thm:obs}). It is obtained by lifting the problem to the second energy level, where the first Hodge--Carleman estimate applies, and by a compactness-uniqueness argument which rests on (c).
\item In dimension three and for nonlinear forces of subcritical growth, the existence of a compact global attractor which coincides with the unstable manifold of the set of equilibria, has finite fractal dimension, and is bounded in $(H^2\cap H^1_0)\times H^1_0$, together with a generalized exponential attractor (Theorem~\ref{thm:main}). We do not prove that the energy is a strict Lyapunov function on the whole phase space. Instead, the attractor is obtained from the criterion of Chueshov for asymptotically smooth systems with a possibly non-strict Lyapunov function \cite[Theorem~2.4.21]{Ch15}, and the Lyapunov property is only exploited along complete trajectories, on which the quasi-stability estimate provides the additional regularity that makes the potential $V$ bounded, so that (c) applies. The two ingredients of this scheme which we could not locate in the literature, the regularity of complete trajectories in bounded forward invariant sets before the attractor is known to exist and the identification of the constant-energy set with the set of equilibria, are proved in Section~\ref{sec:dynamics}; the same scheme is followed in \cite{CMMS26} for scalar waves.
\end{enumerate}

It is worth mentioning two limitations of the present approach. First, the requirement that $\omega$ contains a neighbourhood of $\partial M$ stems from the absence of boundary conditions for $\delta u$ and $\dd u$ in \eqref{eq:intro-P}, and it removes by assumption the phenomena which are specific to the boundary of an elastic body, namely the reflection of rays and the conversion between pressure and shear modes \cite{BLe01}. In Euclidean domains with constant coefficients this hypothesis is more restrictive than the multiplier geometry of \cite{Zua90,Gue00,AC02}, where $\omega$ need only contain a neighbourhood of a part of the boundary. With variable coefficients, the multiplier method of \cite{MNY21} also requires $\omega$ to contain a neighbourhood of the whole boundary, under a one-sided radial condition on the coefficients and with no convexity requirement on the undamped region, so that neither hypothesis contains the other. The contribution of the present paper lies in the Riemannian setting with variable coefficients, not in the geometry of the damping region. The same device, a cut-off vanishing in a neighbourhood of the boundary where the solution is observed, is used in the Carleman estimate of \cite[Lemma~2.3]{BIY08}. Carleman estimates for the Lam\'e system up to a Dirichlet boundary, for solutions which do not vanish near the boundary, were obtained by Imanuvilov and Yamamoto in two-dimensional \cite{IY05} and three-dimensional \cite{IY05b} Euclidean domains, through a microlocal factorization near the boundary and under a sign condition on the weight on the undamped part of the boundary. Their extension to Riemannian manifolds, or to variable coefficients in higher dimension, is an open problem, and (G1) is the price we pay for avoiding it. Second, the $L^2$ estimate of Theorem~\ref{thm:HC0} does not control the gradient of the solution. It is the counterpart of the Euclidean estimates of \cite[Theorem~3.1]{IIY03} and \cite[Theorem~2.3]{IY05}, which present the same loss. In two- and three-dimensional Euclidean domains, the gradient is recovered from an $L^2$ source in \cite[Theorem~2.2]{IY05} and \cite[Theorem~20.2]{IY05b}, at the price of one power of the large parameter, by the boundary microlocal analysis mentioned above, which is not available here. The loss is harmless for bounded potentials. For sources of critical growth, we show in Remark~\ref{rem:critical} that all the steps of the proof of Theorem~\ref{thm:main} remain valid except asymptotic smoothness, which in \cite{CLT08} rests on the absorption of a critical potential by the Carleman estimate and in \cite{CMMS26} on Strichartz estimates. Neither tool is available for the Lam\'e system under \Gh{}, and the critical case remains open.

The paper is organized as follows. In Section~\ref{sec:lame} we derive the three forms of the Lam\'e operator, the Hodge structure and the functional framework. Section~\ref{sec:wellposed} contains the assumptions and the well-posedness of \eqref{eq:main}. In Section~\ref{sec:geometry} we introduce the geometric hypothesis, we construct damping regions of small volume, and we recall the Carleman weights. The Hodge--Carleman estimates are proved in Section~\ref{sec:HC}. Sections~\ref{sec:ucp}, \ref{sec:obs} and \ref{sec:dynamics} are devoted to unique continuation, observability and long-time dynamics, respectively. In Appendix~\ref{app:rays} we compare \Gh{} with the geometric control and obstacle conditions, and in Appendix~\ref{app:IIY} we reproduce, for the reader's convenience, the proof of the Carleman estimate in negative norms of \cite{IIY03} in the present setting.

%=====================================================================
\section{The Lam\'e operator on a Riemannian manifold}\label{sec:lame}
%=====================================================================

\subsection{Notation}
Throughout the paper $(M,g)$ is a compact connected Riemannian manifold of dimension $n\ge2$ with smooth nonempty boundary $\partial M$, $\inte M=M\setminus\partial M$, and $\nu$ is the outward unit normal. We denote by $\nabla$ the Levi-Civita connection, by $\langle\cdot,\cdot\rangle$ and $|\cdot|$ the inner product and norm induced by $g$ on tensors and forms, and by $\dd v$ the Riemannian measure. The curvature tensor is $R(X,Y)=\nabla_X\nabla_Y-\nabla_Y\nabla_X-\nabla_{[X,Y]}$, the second covariant derivative is $\nabla^2_{X,Y}=\nabla_X\nabla_Y-\nabla_{\nabla_XY}$, so that
\[
\nabla^2_{X,Y}-\nabla^2_{Y,X}=R(X,Y),
\]
and the Ricci tensor is $\Ric(Y,Z)=\tr\bigl(X\mapsto R(X,Y)Z\bigr)$, regarded also as the endomorphism $\Ric(u)$ defined by $\langle\Ric(u),Y\rangle=\Ric(u,Y)$. With these conventions the unit sphere has $\Ric=(n-1)g$.

Vector fields and $1$-forms are identified through $g$, and we write $\dd u$ and $\delta u$ for the exterior derivative and the codifferential of the $1$-form associated with $u$. Here $\delta$ is the formal adjoint of $\dd$, so that $\delta u=-\dv u$ for $1$-forms. For a function $\varphi$ and a $1$-form $\alpha$ we shall use the product rules
\begin{equation}\label{eq:product}
\delta(\varphi\alpha)=\varphi\,\delta\alpha-\langle\dd\varphi,\alpha\rangle,\qquad \dd(\varphi\alpha)=\dd\varphi\wedge\alpha+\varphi\,\dd\alpha .
\end{equation}
The Hodge Laplacian is $\DHo=\dd\delta+\delta\dd$, which is nonnegative, and the Bochner Laplacian is $\DB=\tr\nabla^2$, which is nonpositive. For a function $\varphi$, $\nabla\varphi$ is its gradient and $\Hess\varphi$ its Hessian. For a vector field $u$ we set
\[
S_u(X,Y)=\langle\nabla_Xu,Y\rangle,\qquad S_u^\top(X,Y)=S_u(Y,X),\qquad \Def u=\tfrac12\bigl(S_u+S_u^\top\bigr),
\]
and, for a vector field $X$, we denote by $(\nabla u)^\top X$ the vector field defined by $\langle(\nabla u)^\top X,Y\rangle=\langle\nabla_Yu,X\rangle$. The divergence of a $(0,2)$-tensor $T$ is $(\dv T)(Y)=\sum_i(\nabla_{e_i}T)(e_i,Y)$, where $\{e_i\}$ is a local orthonormal frame. Since $\tr\Def u=\dv u$, the stress tensor in \eqref{eq:Lame-def} is a symmetric $(0,2)$-tensor and $\Lam u$ is a vector field.

We write $(\cdot,\cdot)$ and $\|\cdot\|$ for the inner product and the norm of $L^2(M)$, for functions, vector fields and forms alike, and $H^k(M;TM)$, $H^1_0(M;TM)$ for the usual Sobolev spaces of vector fields. When no confusion is possible we simply write $L^2$, $H^1_0$ and $H^k$. For functions of $(t,x)$, $\nabla$ always denotes the spatial covariant derivative.

\subsection{The three forms of the Lam\'e operator}

Throughout the paper we make the following assumption on the material.

\begin{assumption}[L]\label{ass:L}
The density and the Lam\'e moduli satisfy
\[
\varrho,\lambda,\mu\in C^\infty(M),\qquad \varrho>0,\qquad \mu>0\qquad\text{and}\qquad n\lambda+2\mu>0\quad\text{on } M.
\]
\end{assumption}

In dimension three, \Lh{} means that the density, the shear modulus and the bulk modulus are positive. Since $n\ge2$, \Lh{} implies $\lambda>-2\mu/n\ge-\mu$, hence
\begin{equation}\label{eq:lambda-mu}
\lambda+\mu>0\qquad\text{and}\qquad \lambda+2\mu>\mu>0\qquad\text{on } M.
\end{equation}
We define the shear and pressure speeds and the corresponding ray metrics by
\begin{equation}\label{eq:speeds}
c_S=\Bigl(\frac{\mu}{\varrho}\Bigr)^{1/2},\qquad c_P=\Bigl(\frac{\lambda+2\mu}{\varrho}\Bigr)^{1/2},\qquad g_S=c_S^{-2}g,\qquad g_P=c_P^{-2}g,
\end{equation}
and the first-order operator
\begin{equation}\label{eq:B}
\cB u=(\dv u)\nabla\lambda+\nabla_{\nabla\mu}u+(\nabla u)^\top\nabla\mu .
\end{equation}
Since
\[
2\Def u(\nabla\mu,Y)=\langle\nabla_{\nabla\mu}u,Y\rangle+\langle\nabla_Yu,\nabla\mu\rangle,
\]
we have $\cB u=(\dv u)\nabla\lambda+2(\Def u)(\nabla\mu,\cdot)$.

\begin{lemma}\label{lem:ricci-weitz}
Let $u$ be a vector field of class $C^2$. Then
\begin{align}
\dv S_u^\top&=\nabla\dv u+\Ric(u),\label{eq:ricci-id}\\
\DHo u&=-\DB u+\Ric(u).\label{eq:weitz}
\end{align}
\end{lemma}

\begin{proof}
Fix $Y\in T_xM$ and a local orthonormal frame $\{e_i\}$. Since $(\nabla_{e_i}S_u)(X,Y)=\langle\nabla^2_{e_i,X}u,Y\rangle$, we have
\[
(\dv S_u^\top)(Y)=\sum_i\langle\nabla^2_{e_i,Y}u,e_i\rangle=\sum_i\langle\nabla^2_{Y,e_i}u,e_i\rangle+\sum_i\langle R(e_i,Y)u,e_i\rangle .
\]
The first sum equals $Y(\dv u)$, because $\dv u=\tr\nabla u$, and the second one equals $\Ric(Y,u)$ by definition. This proves \eqref{eq:ricci-id}.

For \eqref{eq:weitz}, let $\alpha$ be the $1$-form associated with $u$. Since $\dd\alpha(X,Y)=(\nabla_X\alpha)(Y)-(\nabla_Y\alpha)(X)$ and $\delta\alpha=-\sum_i(\nabla_{e_i}\alpha)(e_i)$, we obtain
\[
(\delta\dd\alpha)(Y)=-\sum_i(\nabla^2_{e_i,e_i}\alpha)(Y)+\sum_i(\nabla^2_{e_i,Y}\alpha)(e_i),\qquad (\dd\delta\alpha)(Y)=-\sum_i(\nabla^2_{Y,e_i}\alpha)(e_i).
\]
We add both identities and recall that the curvature acts on $1$-forms by
\[
(R(X,Y)\alpha)(Z)=-\alpha\bigl(R(X,Y)Z\bigr).
\]
We find
\[
(\DHo\alpha)(Y)=-(\DB\alpha)(Y)-\sum_i\langle u,R(e_i,Y)e_i\rangle=-(\DB\alpha)(Y)+\sum_i\langle R(e_i,Y)u,e_i\rangle,
\]
where the skew-symmetry $\langle R(X,Y)Z,W\rangle=-\langle R(X,Y)W,Z\rangle$ was used in the last step. The last sum is $\Ric(Y,u)$, and \eqref{eq:weitz} follows.
\end{proof}

\begin{proposition}\label{prop:three-forms}
Assume \Lh{}. For every vector field $u$ of class $C^2$,
\begin{align}
\Lam u&=\mu\DB u+(\lambda+\mu)\nabla\dv u+\mu\Ric(u)+\cB u\label{eq:form-B}\\
&=-\mu\DHo u-(\lambda+\mu)\dd\delta u+2\mu\Ric(u)+\cB u\label{eq:form-H}\\
&=-(\lambda+2\mu)\dd\delta u-\mu\,\delta\dd u+2\mu\Ric(u)+\cB u.\label{eq:form-PS}
\end{align}
Since all the operators involved are differential operators with smooth coefficients, these identities also hold in the sense of distributions for $u\in L^2_{\mathrm{loc}}(\inte M;TM)$.
\end{proposition}

\begin{proof}
For a function $\varphi$ we have $\dv(\varphi g)=\nabla\varphi$, and for a symmetric $(0,2)$-tensor $T$ we have $\dv(\mu T)=\mu\dv T+T(\nabla\mu,\cdot)$. Moreover $\dv S_u=\DB u$ and, by \eqref{eq:ricci-id}, $\dv S_u^\top=\nabla\dv u+\Ric(u)$. Hence
\[
\begin{aligned}
\Lam u&=\lambda\nabla\dv u+(\dv u)\nabla\lambda+\mu\bigl(\dv S_u+\dv S_u^\top\bigr)+2(\Def u)(\nabla\mu,\cdot)\\
&=\mu\DB u+(\lambda+\mu)\nabla\dv u+\mu\Ric(u)+\cB u,
\end{aligned}
\]
which is \eqref{eq:form-B}. Since $\nabla\dv u$ corresponds to $-\dd\delta u$, \eqref{eq:form-H} follows from \eqref{eq:form-B} and \eqref{eq:weitz}, and \eqref{eq:form-PS} follows from \eqref{eq:form-H} and $\DHo=\dd\delta+\delta\dd$.
\end{proof}

\subsection{The Hodge structure}

Let $\PP=\partial_t^2-\varrho^{-1}\Lam$. By \eqref{eq:form-PS} and \eqref{eq:speeds},
\begin{equation}\label{eq:P-hodge}
\PP u=\partial_t^2u+c_S^2\DHo u+(c_P^2-c_S^2)\,\dd\delta u+\mathcal Q u,\qquad \mathcal Qu=-2c_S^2\Ric(u)-\varrho^{-1}\cB u .
\end{equation}

\begin{proposition}\label{prop:hodge}
Assume \Lh{}. There exist first-order differential operators $\mathcal M_0$ and $\mathcal M_2$ with smooth coefficients, acting on vector fields and with values in functions and $2$-forms respectively, such that for every time-dependent vector field $u$ of class $C^3$,
\begin{align}
\delta\,\PP u&=\bigl(\partial_t^2+c_P^2\DHo\bigr)\delta u-\bigl\langle\dd c_P^2+\varrho^{-1}\dd(\lambda+2\mu),\dd\delta u\bigr\rangle\notag\\
&\quad-\bigl\langle\dd c_S^2+\varrho^{-1}\dd\mu,\delta\dd u\bigr\rangle+\mathcal M_0u,\label{eq:hodge-P}\\
\dd\,\PP u&=\bigl(\partial_t^2+c_S^2\DHo\bigr)\dd u+\bigl(\dd c_P^2-\varrho^{-1}\dd\lambda\bigr)\wedge\dd\delta u\notag\\
&\quad+\dd c_S^2\wedge\delta\dd u-\varrho^{-1}\nabla_{\nabla\mu}(\dd u)+\mathcal M_2u.\label{eq:hodge-S}
\end{align}
Since both sides are differential operators with smooth coefficients, \eqref{eq:hodge-P} and \eqref{eq:hodge-S} also hold in the sense of distributions for $u\in L^2_{\mathrm{loc}}((0,T)\times\inte M;TM)$.
\end{proposition}

\begin{proof}
Write $y=\delta u$ and $w=\dd u$. By \eqref{eq:P-hodge}, $\PP u=\partial_t^2u+c_P^2\,\dd y+c_S^2\,\delta w+\mathcal Qu$. In this proof, the symbols $\mathcal M$ and $\mathcal M'$ denote first-order differential operators in $u$ with smooth coefficients which may change from line to line. Since the coefficients do not depend on $t$, $\delta$ and $\dd$ commute with $\partial_t^2$. Since $\delta\dd y=\DHo y$ for the function $y$, $\delta\delta w=0$ and $\dd w=0$, the product rules \eqref{eq:product} give
\[
\delta(c_P^2\dd y)=c_P^2\DHo y-\langle\dd c_P^2,\dd y\rangle,\qquad \delta(c_S^2\delta w)=-\langle\dd c_S^2,\delta w\rangle,
\]
\[
\dd(c_P^2\dd y)=\dd c_P^2\wedge\dd y,\qquad \dd(c_S^2\delta w)=c_S^2\DHo w+\dd c_S^2\wedge\delta w .
\]
The terms $\delta(c_S^2\Ric(u))$ and $\dd(c_S^2\Ric(u))$ are of the form $\mathcal Mu$. By \eqref{eq:product}, $\delta(\varrho^{-1}\cB u)=\varrho^{-1}\delta\cB u+\mathcal Mu$ and $\dd(\varrho^{-1}\cB u)=\varrho^{-1}\dd\cB u+\mathcal M'u$, so that it remains to compute the second-order parts of $\delta\cB u=-\dv\cB u$ and of $\dd\cB u$.

Let $X=\nabla\mu$. We compute at a point $x$, in a local orthonormal frame $\{e_i\}$ with $\nabla e_i=0$ at $x$. First,
\[
\dv\bigl((\dv u)\nabla\lambda\bigr)=\langle\nabla\lambda,\nabla\dv u\rangle+(\dv u)\DB\lambda=-\langle\dd\lambda,\dd y\rangle+\mathcal Mu .
\]
Second, since $\nabla_{e_i}(\nabla_Xu)=\nabla^2_{e_i,X}u+\nabla_{\nabla_{e_i}X}u$ and $\nabla^2_{e_i,X}u=\nabla^2_{X,e_i}u+R(e_i,X)u$,
\[
\dv(\nabla_Xu)=\sum_i\langle\nabla^2_{X,e_i}u,e_i\rangle+\Ric(X,u)+\sum_i\langle\nabla_{\nabla_{e_i}X}u,e_i\rangle=X(\dv u)+\mathcal Mu=-\langle\dd\mu,\dd y\rangle+\mathcal Mu .
\]
Third, since $\langle(\nabla u)^\top X,e_i\rangle=\langle\nabla_{e_i}u,X\rangle$, \eqref{eq:weitz} gives
\[
\dv\bigl((\nabla u)^\top X\bigr)=\langle\DB u,X\rangle+\sum_i\langle\nabla_{e_i}u,\nabla_{e_i}X\rangle=-\langle\dd\mu,\dd y+\delta w\rangle+\mathcal Mu,
\]
because $\DB u=-\DHo u+\Ric(u)=-\dd y-\delta w+\Ric(u)$. Adding the three identities,
\[
\delta\cB u=\langle\dd(\lambda+2\mu),\dd y\rangle+\langle\dd\mu,\delta w\rangle+\mathcal Mu .
\]
Collecting terms, we obtain \eqref{eq:hodge-P}.

For the exterior derivative, recall that $\dd\beta(Y,Z)=(\nabla_Y\beta)(Z)-(\nabla_Z\beta)(Y)$ for a $1$-form $\beta$. Since
\[
\dd\bigl((\dv u)\dd\lambda\bigr)=\dd(\dv u)\wedge\dd\lambda=-\dd y\wedge\dd\lambda,
\]
it remains to consider the last two terms of \eqref{eq:B}. For $\beta(Z)=\langle\nabla_Xu,Z\rangle$ we have
\[
(\nabla_Y\beta)(Z)=\langle\nabla^2_{Y,X}u,Z\rangle+\langle\nabla_{\nabla_YX}u,Z\rangle,
\]
and since $\nabla^2_{Y,X}u=\nabla^2_{X,Y}u+R(Y,X)u$,
\[
\dd\beta(Y,Z)=\langle\nabla^2_{X,Y}u,Z\rangle-\langle\nabla^2_{X,Z}u,Y\rangle+(\mathcal Mu)(Y,Z)=(\nabla_Xw)(Y,Z)+(\mathcal Mu)(Y,Z),
\]
because $w(Y,Z)=\langle\nabla_Yu,Z\rangle-\langle\nabla_Zu,Y\rangle$. For $\gamma(Z)=\langle\nabla_Zu,X\rangle$ we have
\[
(\nabla_Y\gamma)(Z)=\langle\nabla^2_{Y,Z}u,X\rangle+\langle\nabla_Zu,\nabla_YX\rangle,
\]
so that
\[
\dd\gamma(Y,Z)=\langle R(Y,Z)u,X\rangle+\langle\nabla_Zu,\nabla_YX\rangle-\langle\nabla_Yu,\nabla_ZX\rangle
\]
is of the form $\mathcal Mu$. Therefore $\dd\cB u=-\dd y\wedge\dd\lambda+\nabla_{\nabla\mu}w+\mathcal Mu$. Collecting terms and using $-\varrho^{-1}(-\dd y\wedge\dd\lambda)=-\varrho^{-1}\dd\lambda\wedge\dd y$, we obtain \eqref{eq:hodge-S}.
\end{proof}

\begin{remark}\label{rem:PS}
Proposition~\ref{prop:hodge} shows that the dilatation $\delta u$ travels with the pressure speed $c_P$ and the vorticity $\dd u$ with the shear speed $c_S$, the couplings between $\delta u$, $\dd u$ and $u$ being of first order. When $\varrho$, $\lambda$ and $\mu$ are constant, \eqref{eq:hodge-P} and \eqref{eq:hodge-S} reduce to
\[
\delta\,\PP u=\bigl(\partial_t^2+c_P^2\DHo\bigr)\delta u-2c_S^2\,\delta\bigl(\Ric(u)\bigr),\qquad \dd\,\PP u=\bigl(\partial_t^2+c_S^2\DHo\bigr)\dd u-2c_S^2\,\dd\bigl(\Ric(u)\bigr).
\]
In flat space $\Ric=0$ and the decoupling is exact, which recovers the classical equations for $\dv u$ and $\operatorname{curl}u$ recalled in \cite[Section~2.1]{BSMS23}. In the Euclidean case with variable coefficients, the first-order couplings are those of the principally diagonal system used in \cite{EINT02,IIY03}.
\end{remark}

\subsection{Green's formula, Korn's inequality and the elastic operator}

We introduce the bilinear form
\begin{equation}\label{eq:form-a}
a(u,v)=\int_M\bigl(\lambda\,\dv u\,\dv v+2\mu\langle\Def u,\Def v\rangle\bigr)\,\dd v,\qquad u,v\in H^1_0(M;TM),
\end{equation}
and the weighted inner product $(u,v)_\varrho=\int_M\varrho\langle u,v\rangle\,\dd v$ on $L^2(M;TM)$, whose norm $\|\cdot\|_\varrho$ is equivalent to $\|\cdot\|$.

\begin{lemma}\label{lem:green}
For $u\in H^2(M;TM)$ and $v\in H^1_0(M;TM)$ we have $(-\Lam u,v)=a(u,v)$.
\end{lemma}

\begin{proof}
By density we may assume that $v$ is smooth with compact support in $\inte M$. Let $X$ be the vector field defined by $\langle X,Y\rangle=\sigma(u)(v,Y)$. Since $\sigma(u)$ is symmetric, a computation in an orthonormal frame gives
\[
\dv X=(\dv\sigma(u))(v)+\sum_i\sigma(u)(\nabla_{e_i}v,e_i)=\langle\Lam u,v\rangle+\langle\sigma(u),\Def v\rangle .
\]
The divergence theorem yields $\int_M\dv X\,\dd v=0$, and
\[
\langle\sigma(u),\Def v\rangle=\lambda\dv u\dv v+2\mu\langle\Def u,\Def v\rangle,
\]
because $\tr\Def v=\dv v$.
\end{proof}

\begin{lemma}\label{lem:korn}
Assume \Lh{}. Then the following assertions hold.
\begin{enumerate}[label=(\roman*)]
\item For every $u\in H^1_0(M;TM)$,
\begin{equation}\label{eq:korn-identity}
\|\nabla u\|^2=2\|\Def u\|^2-\|\dv u\|^2+\int_M\Ric(u,u)\,\dd v .
\end{equation}
\item There exists $C_K>0$ such that $\|\nabla u\|^2\le C_K\|\Def u\|^2$ for every $u\in H^1_0(M;TM)$.
\item Let $c_0=\min_M\min\{2\mu,\,n\lambda+2\mu\}>0$. For every symmetric endomorphism $E$ of $T_xM$, $x\in M$,
\begin{equation}\label{eq:pointwise}
\lambda(x)(\tr E)^2+2\mu(x)|E|^2\ge c_0|E|^2 ,
\end{equation}
and, for every $u\in H^1_0(M;TM)$,
\begin{equation}\label{eq:coercive}
a(u,u)\ge c_0\|\Def u\|^2\ge \frac{c_0}{C_K}\|\nabla u\|^2 .
\end{equation}
In particular, $a$ is an inner product on $H^1_0(M;TM)$ equivalent to the standard one, and
\begin{equation}\label{eq:lambda1}
\lambda_1=\inf_{u\in H^1_0\setminus\{0\}}\frac{a(u,u)}{\|u\|_\varrho^2}>0 .
\end{equation}
\end{enumerate}
\end{lemma}

\begin{proof}
(i) Both sides of \eqref{eq:korn-identity} are continuous on $H^1_0$, so we may assume that $u$ is smooth with compact support in $\inte M$. Since $|\Def u|^2=\tfrac12|\nabla u|^2+\tfrac12\langle S_u,S_u^\top\rangle$, it suffices to show that
\[
\int_M\langle S_u,S_u^\top\rangle\,\dd v=\|\dv u\|^2-\int_M\Ric(u,u)\,\dd v .
\]
In index notation, two integrations by parts without boundary terms and \eqref{eq:ricci-id} give
\[
\begin{aligned}
\int_M\nabla_iu_j\,\nabla^ju^i\,\dd v&=-\int_Mu_j\,\nabla_i\nabla^ju^i\,\dd v=-\int_Mu_j\bigl(\nabla^j\dv u+R^{j}{}_{k}u^k\bigr)\dd v\\
&=\|\dv u\|^2-\int_M\Ric(u,u)\,\dd v .
\end{aligned}
\]

(ii) Extend $(M,g)$ across $\partial M$ by a collar to a Riemannian manifold $(\widetilde M,\tilde g)$, and let $\Omega\subset\widetilde M$ be a connected open set with smooth boundary such that $M\subset\Omega$ and $\Omega\setminus M$ has nonempty interior. The extension $\tilde u$ of $u\in H^1_0(M;TM)$ by zero belongs to $H^1(\Omega;T\widetilde M)$, satisfies $\mathcal L_{\tilde u}\tilde g=2\Def\tilde u$, and vanishes on the open set $\Omega\setminus M$, whose Hausdorff dimension is larger than $n-2$. Hence Chen and Jost \cite[Theorems~1.3 and~5.3]{CJ02} provide a constant $C$ with
\[
\|\nabla\tilde u\|_{L^2(\Omega)}^2\le C\|\mathcal L_{\tilde u}\tilde g\|_{L^2(\Omega)}^2,
\]
which is (ii) with $C_K=4C$. The manifolds of \cite{CJ02} are oriented, and orientability is used in the proof of \cite[Theorem~1.1]{CJ02}. If $\Omega$ is not orientable, let $\pi\colon\widehat M\to\widetilde M$ be the orientation double cover of $\widetilde M$, endowed with the pulled-back metric, for which $\pi$ is a local isometry, and let $\widehat\Omega=\pi^{-1}(\Omega)$, which is an open subset of the oriented manifold $\widehat M$ with smooth boundary, connected because $\Omega$ is not orientable. The lift $\hat u$ of $\tilde u$, defined by $\dd\pi(\hat u(p))=\tilde u(\pi(p))$, belongs to $H^1(\widehat\Omega;T\widehat M)$, vanishes on the open set $\pi^{-1}(\Omega\setminus M)$ and satisfies $\mathcal L_{\hat u}\pi^*\tilde g=\pi^*(\mathcal L_{\tilde u}\tilde g)$. The inequality above holds on $\widehat\Omega$, and each of its two sides equals twice the corresponding integral over $\Omega$, so that it holds on $\Omega$ as well.

(iii) Let $E_0=E-\tfrac1n(\tr E)\,\mathrm{Id}$ be the trace-free part of $E$. Since $|E|^2=|E_0|^2+\tfrac1n(\tr E)^2$,
\[
\lambda(\tr E)^2+2\mu|E|^2=2\mu|E_0|^2+\frac{n\lambda+2\mu}{n}(\tr E)^2\ge c_0|E|^2 .
\]
Applying \eqref{eq:pointwise} with $E=\Def u$ and using (ii) we obtain \eqref{eq:coercive}. The Poincar\'e inequality completes the proof.
\end{proof}

Let $A=-\varrho^{-1}\Lam$ with domain $D(A)=H^2(M;TM)\cap H^1_0(M;TM)$, regarded as an operator in the Hilbert space $L^2_\varrho=(L^2(M;TM),(\cdot,\cdot)_\varrho)$.

\begin{lemma}\label{lem:operatorA}
Assume \Lh{}. Then $A$ is a positive self-adjoint operator in $L^2_\varrho$ with compact resolvent, $D(A^{1/2})=H^1_0(M;TM)$ and $\|A^{1/2}u\|_\varrho^2=a(u,u)$. Moreover, there exists $C>0$ such that $\|u\|_{H^2}\le C\|Au\|$ for every $u\in D(A)$.
\end{lemma}

\begin{proof}
By Lemma~\ref{lem:korn}, $a$ is a closed, symmetric and coercive form on $H^1_0$, densely defined in $L^2_\varrho$. Let $A_a$ be the associated positive self-adjoint operator in $L^2_\varrho$, so that $D(A_a^{1/2})=H^1_0$ and $(A_au,v)_\varrho=a(u,v)$. By Lemma~\ref{lem:green}, $(Au,v)_\varrho=(-\Lam u,v)=a(u,v)$ for $u\in D(A)$ and $v\in H^1_0$, hence $A\subset A_a$. Conversely, if $u\in D(A_a)$, then $u\in H^1_0$ is a weak solution of $-\Lam u=\varrho A_au$ with homogeneous Dirichlet condition. The principal symbol of $-\Lam$ is $\mu|\xi|^2\mathrm{Id}+(\lambda+\mu)\,\xi\otimes\xi$, and by \eqref{eq:lambda-mu}
\[
\mu|\xi|^2|w|^2+(\lambda+\mu)\langle\xi,w\rangle^2\ge\Bigl(\min_M\mu\Bigr)|\xi|^2|w|^2,
\]
so that $-\Lam$ is strongly elliptic. Hence the Dirichlet problem is a regular elliptic boundary problem \cite[Chapter~5, Proposition~11.10]{Tay23}, and the regularity results \cite[Chapter~5, Propositions~11.14 and~11.15]{Tay23}, which apply to systems acting on sections of a vector bundle, give the following statement, used repeatedly below. If $u\in H^1_0(M;TM)$ is a weak solution of $-\Lam u=F$ with $F\in H^k(M;TM)$, $k\ge0$, then $u\in H^{k+2}(M;TM)$ and $\|u\|_{H^{k+2}}\le C(\|F\|_{H^k}+\|u\|)$. For $k=0$ this gives $u\in H^2$ together with the estimate $\|u\|_{H^2}\le C(\|A_au\|+\|u\|)\le C'\|A_au\|$. Thus $A=A_a$. The compactness of the resolvent follows from the compact embedding $H^1_0\subset L^2$.
\end{proof}

\subsection{Elliptic estimates for compactly supported fields}

\begin{lemma}\label{lem:gaffney}
There exist constants $C_1,C_2>0$ such that, for every $u\in H^1(M;TM)$ with compact support in $\inte M$,
\begin{equation}\label{eq:gaffney1}
\|\nabla u\|^2\le\|\dd u\|^2+\|\delta u\|^2+C_1\|u\|^2,
\end{equation}
and, if moreover $u\in H^2(M;TM)$,
\begin{equation}\label{eq:gaffney2}
\|\nabla^2u\|^2\le\|\DB u\|^2+C_2\|u\|_{H^1}^2\le C_2\bigl(\|\nabla\dd u\|^2+\|\nabla\delta u\|^2+\|u\|_{H^1}^2\bigr).
\end{equation}
\end{lemma}

\begin{proof}
By density we may assume that $u$ is smooth with compact support in $\inte M$. Integrating \eqref{eq:weitz} against $u$, we get
\[
\|\dd u\|^2+\|\delta u\|^2=(\DHo u,u)=(-\DB u,u)+\int_M\Ric(u,u)\,\dd v=\|\nabla u\|^2+\int_M\Ric(u,u)\,\dd v,
\]
which gives \eqref{eq:gaffney1} with $C_1=\max_M\|\Ric\|$. For \eqref{eq:gaffney2}, the Ricci identities give, in index notation,
\[
\nabla^i\nabla_i\nabla_ju_k=\nabla_j\nabla^i\nabla_iu_k+\mathcal E_{jk},\qquad |\mathcal E|\le C\bigl(|\nabla u|+|u|\bigr),
\]
where $\mathcal E$ is a sum of contractions of the curvature tensor with $\nabla u$ and of its covariant derivative with $u$. Integrating by parts twice, without boundary terms,
\[
\|\nabla^2u\|^2=-\int_M\nabla_ju_k\,\nabla^i\nabla_i\nabla^ju^k\,\dd v=\|\DB u\|^2-\int_M\nabla^ju^k\,\mathcal E_{jk}\,\dd v\le\|\DB u\|^2+C\|u\|_{H^1}^2 .
\]
Finally, by \eqref{eq:weitz}, $\DB u=-\dd\delta u-\delta\dd u+\Ric(u)$. Since $\dd\delta u$ is the gradient of $\delta u$ and $\delta\dd u$ is a contraction of $\nabla\dd u$, we obtain $\|\DB u\|\le C(\|\nabla\delta u\|+\|\nabla\dd u\|+\|u\|)$, and \eqref{eq:gaffney2} follows.
\end{proof}

The next estimate uses the Hodge form \eqref{eq:form-PS} in an essential way.

\begin{lemma}\label{lem:interior-elliptic}
Assume \Lh{}. There exists $C>0$ such that $\|\nabla^2u\|\le C\bigl(\|\Lam u\|+\|u\|_{H^1}\bigr)$ for every $u\in H^2(M;TM)$ with compact support in $\inte M$.
\end{lemma}

\begin{proof}
By density we may assume that $u$ is smooth with compact support in $\inte M$. Let $y=\delta u$, $w=\dd u$ and $k=(\lambda+2\mu)/\mu$, so that $k>1$ on $M$ by \eqref{eq:lambda-mu}. By \eqref{eq:form-PS},
\[
k\,\dd y+\delta w=-\mu^{-1}\bigl(\Lam u-2\mu\Ric(u)-\cB u\bigr),\qquad\text{hence}\qquad \|k\,\dd y+\delta w\|\le C\bigl(\|\Lam u\|+\|u\|_{H^1}\bigr).
\]
Since $(\dd y,\delta w)=(\dd\dd y,w)=0$ and, by \eqref{eq:product},
\[
(k\,\dd y,\delta w)=(\dd(ky),\delta w)-(y\,\dd k,\delta w)=-(y\,\dd k,\delta w),
\]
we have
\[
\|\dd y\|^2+\|\delta w\|^2\le\|k\,\dd y\|^2+\|\delta w\|^2=\|k\,\dd y+\delta w\|^2+2(y\,\dd k,\delta w).
\]
Moreover,
\[
\|\DHo u\|^2=\|\dd y+\delta w\|^2=\|\dd y\|^2+\|\delta w\|^2,\qquad |2(y\,\dd k,\delta w)|\le C\|\nabla u\|\,\|\nabla^2u\|,
\]
and $\|\DB u\|\le\|\DHo u\|+C\|u\|$ by \eqref{eq:weitz}. Inserting these relations in the first inequality of \eqref{eq:gaffney2}, we obtain
\[
\|\nabla^2u\|^2\le C\bigl(\|\Lam u\|^2+\|u\|_{H^1}^2\bigr)+C\|\nabla u\|\,\|\nabla^2u\|,
\]
and the Young inequality completes the proof.
\end{proof}

%=====================================================================
\section{Assumptions and well-posedness}\label{sec:wellposed}
%=====================================================================

Let $\omega\subset M$ be an open set, which plays the role of the damping region.
We first state the geometric hypothesis on $\omega$, which is formulated in terms of the ray metrics of the medium. For $c\in\{c_S,c_P\}$, we denote by $g_c=c^{-2}g$ the corresponding ray metric \eqref{eq:speeds} and by $\nabla^2_{g_c}$ the Hessian with respect to $g_c$.

\begin{assumption}[G]\label{ass:G}
The open set $\omega\subset M$ and a function $d\in C^\infty(M)$ satisfy the following conditions, where $K=M\setminus\omega$.
\begin{enumerate}[label=(G\arabic*)]
\item $\partial M\subset\omega$, that is, $K$ is a compact subset of $\inte M$.
\item $d$ has no critical points on $K$ and is strictly convex on $K$ with respect to both ray metrics, namely
\[
\nabla d(x)\ne0\quad\text{and}\quad\nabla^2_{g_S}d(x)(X,X)>0,\quad \nabla^2_{g_P}d(x)(X,X)>0\qquad\text{for } x\in K,\ X\in T_xM\setminus\{0\}.
\]
\end{enumerate}
\end{assumption}

The set $K$ may be empty, in which case $\omega=M$ and (G2) is void, and it need not be connected; (G2) only concerns the restriction of $d$ to $K$. Assumption~\ref{ass:G} is discussed in Section~\ref{sec:geometry}, where we also show that it can be fulfilled with damping regions of arbitrarily small volume. It is not needed for the well-posedness results of this section.

We now state the conditions on the damping, the load and the nonlinearity.
\begin{enumerate}[label=(H\arabic*)]
\item $a\in L^\infty(M)$, $a\ge0$ a.e. in $M$, and there exists $a_0>0$ such that $a\ge a_0$ a.e. in $\omega$.
\item $h\in L^2(M;TM)$.
\item $n=3$ and $F\in C^2(TM)$. For $x\in M$, $f(x,u)\in T_xM$ denotes the gradient of $F(x,\cdot)$ at $u$ with respect to $g_x$, and $\partial_uf(x,u)\in\operatorname{End}(T_xM)$ the corresponding Hessian. There exist $\gamma\in[0,2)$ and $C_f\ge0$ such that
\[
\|\partial_uf(x,u)\|\le C_f\bigl(1+|u|^\gamma\bigr)\qquad\text{for all } (x,u)\in TM.
\]
\item There exist $\kappa\in[0,\lambda_1)$ and $C_F\ge0$ such that, for all $(x,u)\in TM$,
\[
F(x,u)\ge-\frac\kappa2\varrho(x)|u|^2-C_F\qquad\text{and}\qquad \langle f(x,u),u\rangle\ge-\kappa\varrho(x)|u|^2-C_F .
\]
\end{enumerate}

A typical example is $F(x,u)=\frac{1}{\gamma+2}|u|^{\gamma+2}-\frac{m(x)\varrho(x)}{2}|u|^2$ with $\gamma\in[0,2)$ and $m\in C^\infty(M)$ satisfying $\max_Mm<\lambda_1$. Since $M$ is compact, $f_0=\sup_{x\in M}|f(x,0)|$ is finite, and the mean value theorem applied on the segment joining $u$ and $v$ in $T_xM$ gives
\begin{equation}\label{eq:f-lip}
|f(x,u)-f(x,v)|\le C_f\bigl(1+|u|^\gamma+|v|^\gamma\bigr)|u-v|,\qquad |f(x,u)|\le f_0+C_f\bigl(|u|+|u|^{\gamma+1}\bigr).
\end{equation}
The restriction $n=3$ in (H3) enters only through the Sobolev embeddings $H^1_0\subset L^6$, in Lemma~\ref{lem:nemytskii} and in the proof of Lemma~\ref{lem:stab}, and $H^2\subset C(M)$, in Proposition~\ref{prop:reg} and Lemma~\ref{lem:rigidity}. For $n=2$ all the results of Section~\ref{sec:dynamics} hold with nonlinearities of arbitrary polynomial growth, with the obvious changes. For $n\ge4$ the method breaks down in Lemma~\ref{lem:rigidity}, because the boundedness of the potential generated by the trajectory, which Theorem~\ref{thm:ucp} requires, is no longer available. For a vector field $u$ we write $\hat f(u)=f(\cdot,u(\cdot))$, and we set
\begin{equation}\label{eq:r}
r=\frac{6}{3-\gamma}\in[2,6).
\end{equation}

\begin{lemma}\label{lem:nemytskii}
Assume (H3). Then $\hat f$ maps $H^1_0(M;TM)$ into $L^2(M;TM)$, and there exists $C>0$ such that
\[
\|\hat f(u)-\hat f(v)\|\le C\bigl(1+\|u\|_{L^6}^\gamma+\|v\|_{L^6}^\gamma\bigr)\|u-v\|_{L^r}\qquad\text{for all } u,v\in H^1_0(M;TM).
\]
\end{lemma}

\begin{proof}
Since $\frac12=\frac\gamma6+\frac1r$, the claim follows from \eqref{eq:f-lip}, the H\"older inequality and the embedding $H^1_0\subset L^6$, valid because $n=3$.
\end{proof}

The phase space is $\cH=H^1_0(M;TM)\times L^2(M;TM)$ with the norm $\|(u,v)\|_{\cH}^2=a(u,u)+\|v\|_\varrho^2$. For $y=(u,v)\in\cH$ we define the energy
\begin{equation}\label{eq:Phi}
\Phi(y)=\frac12\|v\|_\varrho^2+\frac12a(u,u)+\int_MF(x,u)\,\dd v-(h,u),
\end{equation}
and, for a solution of a linear problem, the quadratic energy $E(t)=\frac12\bigl(\|\partial_tu(t)\|_\varrho^2+a(u(t),u(t))\bigr)$. In $\cH$ we consider the operator $\cA_0(u,v)=(v,-Au-\varrho^{-1}av)$ with domain $D(\cA_0)=D(A)\times H^1_0$, and the map $\mathcal F(u,v)=(0,\varrho^{-1}(h-\hat f(u)))$, so that \eqref{eq:main} reads $\partial_tU=\cA_0U+\mathcal F(U)$ for $U=(u,\partial_tu)$.

\begin{proposition}\label{prop:wellposed}
Assume \Lh{}, (H1)--(H4). For every $y=(u_0,u_1)\in\cH$ problem \eqref{eq:main} has a unique mild solution
\[
u\in C([0,\infty);H^1_0(M;TM))\cap C^1([0,\infty);L^2(M;TM)),
\]
and the energy identity
\begin{equation}\label{eq:energy-identity}
\Phi(u(t),\partial_tu(t))+\int_s^t\bigl(a\,\partial_tu,\partial_tu\bigr)\,\dd\tau=\Phi(u(s),\partial_tu(s)),\qquad 0\le s\le t,
\end{equation}
holds. Setting $S_ty=(u(t),\partial_tu(t))$, we obtain a dynamical system $(\cH,S_t)$. For every $R>0$ there exists $L_R$ such that $\|S_ty^1-S_ty^2\|_{\cH}\le e^{L_Rt}\|y^1-y^2\|_{\cH}$ whenever $\|y^1\|_{\cH},\|y^2\|_{\cH}\le R$.
\end{proposition}

\begin{proof}
For $U=(u,v)\in D(\cA_0)$, Lemma~\ref{lem:operatorA} gives
\[
(\cA_0U,U)_{\cH}=a(v,u)-a(u,v)-(av,v)\le0,
\]
and, given $(f_1,f_2)\in\cH$, the equation $U-\cA_0U=(f_1,f_2)$ reduces to $v=u-f_1$ and
\[
\varrho u-\Lam u+au=\varrho(f_1+f_2)+af_1,
\]
which has a unique solution $u\in D(A)$ by the Lax--Milgram theorem and Lemma~\ref{lem:operatorA}. By the Lumer--Phillips theorem \cite[Theorem~1.4.3]{Pazy83}, $\cA_0$ generates a $C_0$-semigroup of contractions on $\cH$. By Lemma~\ref{lem:nemytskii}, $\mathcal F$ is Lipschitz on bounded subsets of $\cH$. Hence \cite[Theorem~6.1.4]{Pazy83} provides a unique mild solution on a maximal interval $[0,t_{\max})$, with $\|S_ty\|_{\cH}\to\infty$ as $t\to t_{\max}$ if $t_{\max}<\infty$.

If $y\in D(\cA_0)$, we replace $\mathcal F$ by a globally Lipschitz map which coincides with $\mathcal F$ on a ball containing the trajectory on a given compact interval. Then \cite[Theorem~6.1.6]{Pazy83} shows that the solution is a strong solution. For strong solutions,
\[
\frac{\dd}{\dd t}\,\frac12\|U\|^2_{\cH}=(\cA_0U,U)_{\cH}+(\mathcal F(U),U)_{\cH},
\]
and the functional $u\mapsto\int_MF(x,u)\,\dd v$ is of class $C^1$ on $H^1_0$ with derivative $\hat f(u)$, because $\gamma+2<6$. Hence \eqref{eq:energy-identity} holds for strong solutions, and by density and continuous dependence it holds for mild solutions.

By (H4) and \eqref{eq:lambda1}, $\int_MF(x,u)\,\dd v\ge-\frac{\kappa}{2\lambda_1}a(u,u)-C_F|M|$, and therefore
\begin{equation}\label{eq:Phi-lower}
\Phi(y)\ge\frac12\|v\|_\varrho^2+\frac14\Bigl(1-\frac\kappa{\lambda_1}\Bigr)a(u,u)-C,\qquad y=(u,v)\in\cH,
\end{equation}
where the term $(h,u)$ has been absorbed by the Young inequality. Since $\Phi$ is nonincreasing along trajectories by \eqref{eq:energy-identity}, the solution remains bounded and $t_{\max}=\infty$. The Lipschitz estimate follows from the variation of constants formula and the Gronwall lemma, because trajectories issued from the ball of radius $R$ remain in a ball whose radius depends only on $R$.
\end{proof}

\begin{lemma}\label{lem:Phi-equilibria}
Assume \Lh{}, (H1)--(H4). Then $\Phi$ is continuous on $\cH$, bounded on bounded sets, and the sets $\{y\in\cH:\Phi(y)\le R\}$ are bounded and forward invariant for every $R$. The set of stationary points of $(\cH,S_t)$ is
\[
\cN=\bigl\{(e,0):\ e\in D(A),\ -\Lam e+\hat f(e)=h\bigr\},
\]
and it is bounded in $\cH$.
\end{lemma}

\begin{proof}
By \eqref{eq:f-lip}, $|F(x,u)|\le C(1+|u|^{\gamma+2})$ with $\gamma+2<6$, which gives the continuity of $\Phi$ and its boundedness on bounded sets. The boundedness of the sublevel sets follows from \eqref{eq:Phi-lower}, and their forward invariance from \eqref{eq:energy-identity}. If $S_ty=y$ for all $t$, then $u(t)=e$ is constant, $\partial_tu=0$, and the integrated form \eqref{eq:mild-int} of the equation, proved in Section~\ref{sec:dynamics} for every mild solution, gives $y=(e,0)\in D(\cA_0)$ with $-\Lam e+\hat f(e)=h$. Taking the inner product with $e$ and using (H4),
\[
a(e,e)\le\kappa\|e\|_\varrho^2+C_F|M|+\|h\|\|e\|\le\frac\kappa{\lambda_1}a(e,e)+C_F|M|+C\|h\|\,a(e,e)^{1/2},
\]
so that $a(e,e)$ is bounded by a constant depending only on $\kappa$, $\lambda_1$, $C_F$, $|M|$, $\varrho$ and $\|h\|$.
\end{proof}

%=====================================================================
\section{Geometric hypothesis and Carleman weights}\label{sec:geometry}
%=====================================================================

In this section we derive a quantitative form of the convexity required in Assumption~\ref{ass:G}, we show that \Gh{} can be fulfilled by damping regions of arbitrarily small volume, and we introduce the Carleman weights.

\begin{lemma}\label{lem:U}
Assume \Gh{}. There exist open sets $U,U_1\subset M$ and constants $\rho>0$, $\varpi>0$ such that $K\subset U$, $\overline U\subset U_1$, the closures $\overline U$ and $\overline{U_1}$ are compact manifolds with smooth boundary contained in $\inte M$, and, for $c\in\{c_S,c_P\}$,
\begin{equation}\label{eq:convex}
\nabla^2_{g_c}d(x)(X,X)\ge2\rho|X|_{g_c}^2,\qquad |\nabla d(x)|\ge\varpi,\qquad x\in\overline{U_1},\ X\in T_xM .
\end{equation}
\end{lemma}

\begin{proof}
The set $\cO$ of points of $\inte M$ where $\nabla d\ne0$ and both Hessians $\nabla^2_{g_S}d$ and $\nabla^2_{g_P}d$ are positive definite is open and contains the compact set $K$. Let $\theta\in C^\infty_c(\cO)$ with $0\le\theta\le1$ and $\theta=1$ on a neighbourhood of $K$. By Sard's theorem there are regular values $0<\bar c_1<\bar c<1$ of $\theta$, and $U=\{\theta>\bar c\}$, $U_1=\{\theta>\bar c_1\}$ have the required properties, with $\rho$ and $\varpi$ given by the minima of the continuous functions $\nabla^2_{g_c}d(X,X)/2$ and $|\nabla d|$ on the compact sets $\{(x,X):x\in\overline{U_1},\ |X|_{g_c}=1\}$ and $\overline{U_1}$.
\end{proof}

From now on $U$, $U_1$, $\rho$ and $\varpi$ are fixed as in Lemma~\ref{lem:U}. They depend on the auxiliary choices made in its proof, and so does the geometric time defined next; all the constants below are relative to these fixed choices. If $K=\emptyset$, the set $\cO$ in the proof of Lemma~\ref{lem:U} may be empty, and we agree that $U=U_1=\emptyset$, $\rho=\varpi=1$ and $T_*=0$; the Carleman estimates of Section~\ref{sec:HC} are then void, and the proofs of Sections~\ref{sec:ucp} and~\ref{sec:obs} go through with the obvious simplifications, every cut-off supported in $U$ being zero. We define the geometric time
\begin{equation}\label{eq:Tstar}
T_*=2\Bigl(\frac{\max_{\overline U}d-\min_{\overline U}d}{\rho}\Bigr)^{1/2}.
\end{equation}

\begin{remark}\label{rem:examples}
(i) If $\lambda/\mu$ is constant, then $g_P=\frac{\mu}{\lambda+2\mu}\,g_S$ is a constant multiple of $g_S$. Since the Levi-Civita connection is invariant under constant rescalings, $\nabla^2_{g_P}d=\nabla^2_{g_S}d$, and the two convexity conditions in (G2) coincide. This is the case of a medium with constant Poisson ratio, and in particular of a Poisson solid, $\lambda=\mu$, which is a classical model for the rocks of the Earth's crust and mantle \cite[Section~4.6]{Puj03}. If moreover $\varrho$ and $\mu$ are constant, (G2) means that $d$ is strictly convex with respect to $g$. In general (G2) is a condition on the ray metrics and not on $g$, and a function which is strictly convex with respect to $g$ need not satisfy it, as the next example shows.

(ii) Let $M$ be a smooth bounded domain of $\R^n$ with the Euclidean metric, let $x_0\notin K$ and $d(x)=|x-x_0|^2$. We write $g_c=e^{2\varphi_c}g$ with $\varphi_c=-\log c$, so that the Christoffel symbols of $g_c$ are
\[
\Gamma^k_{ij}=\delta_{ik}\partial_j\varphi_c+\delta_{jk}\partial_i\varphi_c-\delta_{ij}\partial_k\varphi_c .
\]
A direct computation gives
\[
\nabla^2_{g_c}d(X,X)=2|X|^2+2|X|^2\langle\nabla\varphi_c,x-x_0\rangle-4\langle X,\nabla\varphi_c\rangle\langle X,x-x_0\rangle .
\]
The quadratic form $2|X|^2\langle a,b\rangle-4\langle X,a\rangle\langle X,b\rangle$ has smallest eigenvalue $-2|a||b|$, and therefore
\[
\nabla^2_{g_c}d(X,X)\ge2\bigl(1-|\nabla\log c|\,|x-x_0|\bigr)|X|^2 .
\]
Hence \Gh{} holds for every open set $\omega\supset\partial M$ with $x_0\notin K$ provided $|x-x_0|\,|\nabla\log c|<1$ on $K$ for $c\in\{c_S,c_P\}$, which is the case of a slowly varying medium. When $c_S$ and $c_P$ are constant, this is a particular case of the multiplier geometry of Zuazua \cite{Zua90}, and of its elastic versions \cite{Gue00,AC02}, where $\omega$ is only required to contain a neighbourhood of the part of $\partial M$ on which $\langle x-x_0,\nu\rangle>0$; in that case \Gh{} is strictly more restrictive.

(iii) Let $\varrho$, $\lambda$ and $\mu$ be constant and let $M$ be a smooth compact domain in a complete simply connected manifold $(N,g)$ with nonpositive sectional curvature. Then $d=\frac12\operatorname{dist}(p,\cdot)^2$ is smooth, $\Hess d\ge g$ by the Hessian comparison theorem \cite[Theorem~6.4.3]{Pet16}, and \Gh{} holds when $\omega$ contains a neighbourhood of $\partial M$ and either $p\notin M$ or $\omega$ contains a neighbourhood of $p$.

(iv) On the other hand, if $K$ contains a closed geodesic $\Gamma$ of $g_S$ or of $g_P$, then \Gh{} fails, because $(d\circ\Gamma)''=\nabla^2_{g_c}d(\Gamma',\Gamma')>0$ is incompatible with the periodicity of $d\circ\Gamma$. Thus $\omega$ must meet every closed ray, in the spirit of the geometric control condition \cite{BLR92}. We refer to Appendix~\ref{app:rays} for a comparison with the geometric control and obstacle conditions, and to \cite[Chapter~6]{BY17} for systematic constructions of convex weights.
\end{remark}

The next result shows that \Gh{} imposes no restriction on the size of the damping region, besides (G1). It is the counterpart for the Lam\'e system of the constructions of Cavalcanti, Domingos Cavalcanti, Fukuoka and Soriano \cite[Section~6]{CDFS10} and of \cite[Section~2]{CMMS26} for the scalar wave equation.

\begin{proposition}\label{prop:small}
Assume \Lh{}. For every $\varepsilon>0$ there exist an open set $\omega\subset M$ with $\vol(\omega)<\varepsilon$ and a function $d\in C^\infty(M)$ such that $(\omega,d)$ satisfies \Gh{}.
\end{proposition}

\begin{proof}
\emph{Step 1 (collar).} For $\ell>0$ let $C_\ell=\{x\in M:\dist(x,\partial M)<\ell\}$. Since the sets $C_\ell$ decrease to $\partial M$ as $\ell\downarrow0$ and $\vol(\partial M)=0$, we may fix $\ell$ with $\vol(C_\ell)<\varepsilon/2$. The set $K'=M\setminus C_\ell$ is a compact subset of $\inte M$.

\emph{Step 2 (local convex functions).} Let $p\in K'$ and let $x=(x^1,\dots,x^n)$ be a coordinate chart defined on a neighbourhood $W_p$ of $p$, with $\overline{W_p}\subset\inte M$ compact, $x(p)=0$ and $x(W_p)\supset\overline{B_{r_0}}$, where $B_r$ denotes the Euclidean ball of radius $r$ centred at $0$. Let $\Lambda_p$ be a bound for the Christoffel symbols of $g_S$ and $g_P$ on $x^{-1}(\overline{B_{r_0}})$, and fix $r_p\in(0,r_0/3]$ with $3n^3\Lambda_pr_p\le\frac12$. Choose $q\in\R^n$ with $|q|=2r_p$, and let $d_p=\frac12|x-q|^2$ on $B_p=x^{-1}(B_{r_p})$. For $c\in\{c_S,c_P\}$ and $X\in T_zM$, $z\in B_p$, writing $\Gamma^k_{ij}$ for the Christoffel symbols of $g_c$,
\[
\nabla^2_{g_c}d_p(X,X)=\sum_{i,j}\Bigl(\delta_{ij}-\sum_k\Gamma^k_{ij}(x^k-q^k)\Bigr)X^iX^j\ge\bigl(1-n^3\Lambda_p|x-q|\bigr)|X|_e^2\ge\frac12|X|_e^2,
\]
because $|x-q|\le3r_p$ on $B_p$, where $|X|_e$ is the Euclidean norm of the components of $X$. Moreover the coordinate gradient of $d_p$ is $x-q$, which does not vanish on $B_p$ because $|x-q|\ge r_p$ there. Thus $\nabla d_p\ne0$ and $d_p$ is strictly convex with respect to $g_S$ and $g_P$ on $B_p$. The set $B_p$ is open, its boundary $\partial B_p=x^{-1}(\partial B_{r_p})$ is a smooth compact hypersurface, and $\overline{B_p}\subset\inte M$.

\emph{Step 3 (disjoint pieces).} By compactness, $K'\subset B_1\cup\dots\cup B_k$ for finitely many sets $B_j=B_{p_j}$ as in Step 2, with functions $d_j=d_{p_j}$. The set $\Sigma=\bigcup_j\partial B_j$ is a finite union of smooth compact hypersurfaces, so that $\vol(\Sigma)=0$, and by the outer regularity of the Riemannian measure there exists an open set $O\supset\Sigma$ with $\vol(O)<\varepsilon/2$. Let
\[
N_j=B_j\setminus\bigcup_{i<j}\overline{B_i},\qquad K=K'\setminus O .
\]
The sets $N_j$ are open and mutually disjoint, and $K$ is compact. If $x\in K$, let $j$ be the smallest index with $x\in B_j$. For $i<j$ we have $x\notin B_i$ and $x\notin\partial B_i\subset\Sigma$, hence $x\notin\overline{B_i}$ and $x\in N_j$. Thus $K\subset\bigcup_jN_j$, and the sets $K_j=K\cap N_j=K\setminus\bigcup_{i\ne j}N_i$ are compact.

\emph{Step 4 (the global function).} Let $\chi_j\in C^\infty_c(N_j)$ with $0\le\chi_j\le1$ and $\chi_j=1$ on a neighbourhood of $K_j$, and let $d=\sum_j\chi_jd_j$, where each product is extended by zero outside $\supp\chi_j\subset B_j$. Then $d\in C^\infty(M)$. Since $\supp\chi_i$ is a compact subset of $N_i$ and $N_i\cap N_j=\emptyset$ for $i\ne j$, the compact sets $K_j$ and $\bigcup_{i\ne j}\supp\chi_i$ are disjoint, so that $d=d_j$ on a neighbourhood of $K_j$. By Step 2, $d$ satisfies (G2) on $K=\bigcup_jK_j$.

\emph{Step 5 (the damping region).} Let $\omega=C_\ell\cup O$. Then $\omega$ is open, $\vol(\omega)<\varepsilon$, $\partial M\subset C_\ell\subset\omega$, and $M\setminus\omega=K'\setminus O=K$. Hence $(\omega,d)$ satisfies \Gh{}.
\end{proof}

\begin{remark}\label{rem:boundary}
In Proposition~\ref{prop:small} the volume of $\omega$ is arbitrarily small, while $\omega$ contains the whole boundary. The constants $\rho$, $\varpi$ and $T_*$ of Lemma~\ref{lem:U} depend on $\varepsilon$, and the construction gives no control of them as $\varepsilon\to0$. Condition (G1) cannot be removed by the method of this paper. The functions $\delta u$ and $\dd u$ in \eqref{eq:intro-P} satisfy no boundary condition, so that the Carleman estimates of Section~\ref{sec:HC} are proved for fields with compact support in $\inte M$. For the scalar wave equation with Dirichlet condition, the boundary term of the Carleman estimate has a favourable sign on the part of $\partial M$ where $\partial_\nu d<0$, see \cite[Theorem~4.2]{BY17}, and this is what allows damping regions whose boundary trace is also arbitrarily small, see \cite{CMMS26}. Carleman estimates for the Lam\'e system up to a Dirichlet boundary are available in two- and three-dimensional Euclidean domains \cite{IY05,IY05b}, under a sign condition on the weight on the undamped part of the boundary, and their Riemannian counterpart is not known. From the physical viewpoint, (G1) describes an elastic body whose boundary layer is made of a dissipative material.
\end{remark}

\subsection{Carleman weights}
Let $T>0$, $\beta\in(0,\rho)$ and $I_T=(-1,T+1)$. We set
\begin{equation}\label{eq:weight}
\psi(t,x)=d(x)-\beta\Bigl(t-\frac T2\Bigr)^2+\beta_0,\qquad \phi(t,x)=e^{\gamma_0\psi(t,x)},\qquad (t,x)\in\overline{I_T}\times\overline{U_1},
\end{equation}
where $\beta_0=1+\max_{\overline{U_1}}|d|+\beta(T+2)^2/4$, so that $\psi\ge1$, and $\gamma_0\ge1$ is the constant given by Lemma~\ref{lem:carleman-scalar} below, which depends on $T$ and $\beta$, as do the constants $s_0$ and $C$ of all the Carleman estimates of this section and of Section~\ref{sec:HC}. In the rest of the paper we use the abbreviation
\[
\ww{F}=\int_{I_T}\!\int_M e^{2s\phi}|F|^2\,\dd v\,\dd t
\]
for functions, vector fields and forms $F$ supported in $I_T\times\overline{U_1}$.

\begin{lemma}\label{lem:carleman-scalar}
Assume \Lh{} and \Gh{}. There exists $\gamma_0\ge1$ with the following property. For $c\in\{c_S,c_P\}$ and for every first-order differential operator $P_1$ with bounded measurable coefficients, there exist $s_0>0$ and $C>0$ such that
\[
s\ww{\nabla u}+s\ww{\partial_tu}+s^3\ww{u}\le C\ww{\partial_t^2u-c^2\Delta_gu+P_1u},\qquad s\ge s_0,
\]
for every function $u\in H^2(I_T\times M)$ with compact support in $I_T\times U_1$. Here $\Delta_g=\DB$ is the Laplace--Beltrami operator.
\end{lemma}

\begin{proof}
For a conformal metric $\hat g=e^{2\varphi}g$ one has
\[
\Delta_{\hat g}u=e^{-2\varphi}\bigl(\Delta_gu+(n-2)\langle\nabla\varphi,\nabla u\rangle\bigr).
\]
With $\hat g=g_c$ and $\varphi=-\log c$ we obtain
\[
\partial_t^2u-c^2\Delta_gu=\partial_t^2u-\Delta_{g_c}u-(n-2)\,c\,\langle\nabla c,\nabla u\rangle .
\]
By \eqref{eq:convex}, $\rho'=(\beta+\rho)/2$ satisfies $\nabla^2_{g_c}d(X,X)>2\rho'|X|^2_{g_c}$ for $X\ne0$, and $|\nabla_{g_c}d|_{g_c}=c|\nabla d|\ge\varpi\min_Mc$ on $\overline{U_1}$. Hence the compact manifold $(\overline{U_1},g_c)$ and the positive function $d+1+\max_{\overline{U_1}}|d|$ satisfy the assumptions called (A.1) and (A.2) in \cite[Section~4.2]{BY17} with the constant $\rho'$, and $0<\beta<\rho'$ because $\beta<\rho$. We apply \cite[Theorem~4.1]{BY17} on the time interval $I_T$, with $t_0=T/2$, to the operator $\partial_t^2-\Delta_{g_c}-(n-2)c\langle\nabla c,\nabla\,\cdot\,\rangle$. Since $u$ vanishes near $\partial U_1$ and near the endpoints of $I_T$, the boundary integral in that theorem vanishes, and we obtain $\gamma_*(c)>0$ such that, for $\gamma_0>\gamma_*(c)$ and $s$ large,
\[
\int_{I_T}\!\int_{U_1}e^{2s\phi}\sigma\bigl(|\nabla_{g_c}u|_{g_c}^2+|\partial_tu|^2+\sigma^2|u|^2\bigr)\dd v_{g_c}\dd t\le C\int_{I_T}\!\int_{U_1}e^{2s\phi}|\partial_t^2u-c^2\Delta_gu|^2\dd v_{g_c}\dd t,
\]
where $\sigma=s\gamma_0\phi$. Here the weight of \cite[(4.2)--(4.3)]{BY17}, with $\psi_0=d+1+\max_{\overline{U_1}}|d|$ and with the constant $\beta(T+2)^2/4$ in place of $\beta_0$, coincides with \eqref{eq:weight}. Since $e^{\gamma_0}\le\phi\le e^{\gamma_0\max\psi}$, $\dd v_{g_c}=c^{-n}\dd v$ and $|\nabla_{g_c}u|^2_{g_c}=c^2|\nabla u|^2$, we obtain the claim for $P_1=0$ with $\gamma_0=1+\max\{\gamma_*(c_S),\gamma_*(c_P)\}$. Finally, $|P_1u|^2\le C(|\nabla u|^2+|\partial_tu|^2+|u|^2)$, and the contribution of $P_1$ to the right-hand side is absorbed by the left-hand side for $s$ large.
\end{proof}

We now extend Lemma~\ref{lem:carleman-scalar} to forms. In local coordinates, the principal symbol of $\DHo$ acting on $k$-forms is $|\xi|^2\mathrm{Id}$, because $\varepsilon_\xi\iota_\xi+\iota_\xi\varepsilon_\xi=|\xi|^2$ for the exterior and interior multiplications by $\xi$. Hence, for every $k$-form $\zeta$,
\begin{equation}\label{eq:DH-coord}
(\DHo\zeta)_I=-\Delta_g(\zeta_I)+\sum_J\bigl(b_{IJ}^j\partial_j\zeta_J+c_{IJ}\zeta_J\bigr)
\end{equation}
in every coordinate chart, with smooth coefficients $b_{IJ}^j$ and $c_{IJ}$.

\begin{lemma}\label{lem:carleman-forms}
Assume \Lh{} and \Gh{}. Let $k\in\{0,1,2\}$, $c\in\{c_S,c_P\}$, and let $\mathcal T$ be a first-order differential operator in $(t,x)$ on $k$-forms with bounded measurable coefficients. There exist $s_0>0$ and $C>0$ such that
\[
s\ww{\nabla\zeta}+s\ww{\partial_t\zeta}+s^3\ww{\zeta}\le C\ww{\partial_t^2\zeta+c^2\DHo\zeta+\mathcal T\zeta},\qquad s\ge s_0,
\]
for every $k$-form $\zeta\in H^2((0,T)\times M)$ with compact support in $(0,T)\times U$.
\end{lemma}

\begin{proof}
Cover $\overline U$ by finitely many open sets $W_1,\dots,W_J$ with $\overline{W_m}\subset W_m'$ and $\overline{W_m'}\subset U_1$, where each $\overline{W_m'}$ is contained in the domain $W_m''$ of a coordinate chart, which we use on $W_m$ and on $W_m'$. Let $\chi_1,\dots,\chi_J$ be smooth functions with $\supp\chi_m\subset W_m$ and $\sum_m\chi_m=1$ on a neighbourhood of $\overline U$. For each $m$ and each multi-index $I$, the function $u_{m,I}=\chi_m\zeta_I$, extended by zero, belongs to $H^2((0,T)\times M)$ and has compact support in $(0,T)\times U$. Writing $\PP_c\zeta=\partial_t^2\zeta+c^2\DHo\zeta+Q\zeta$, \eqref{eq:DH-coord} yields
\[
\bigl|\bigl(\partial_t^2-c^2\Delta_g\bigr)u_{m,I}-\chi_m(\PP_c\zeta)_I\bigr|\le C\bigl(|\partial_t\zeta|+|\nabla\zeta|+|\zeta|\bigr),
\]
where the commutator $[\Delta_g,\chi_m]$ and the lower-order terms in \eqref{eq:DH-coord} and in $\mathcal T$ have been estimated by the pointwise norms of $\partial_t\zeta$, $\nabla\zeta$ and $\zeta$, which are equivalent to the coordinate norms on $\supp\chi_m$. Lemma~\ref{lem:carleman-scalar} with $P_1=0$ gives
\[
s\ww{\nabla u_{m,I}}+s\ww{\partial_tu_{m,I}}+s^3\ww{u_{m,I}}\le C\ww{\PP_c\zeta}+C\ww{\partial_t\zeta}+C\ww{\nabla\zeta}+C\ww{\zeta}.
\]
Since $\sum_m\chi_m=1$ on the support of $\zeta$, we have $|\zeta|^2\le C\sum_{m,I}|u_{m,I}|^2$, $|\partial_t\zeta|^2\le C\sum_{m,I}|\partial_tu_{m,I}|^2$ and $|\nabla\zeta|^2\le C\sum_{m,I}\bigl(|\nabla u_{m,I}|^2+|\zeta|^2\bigr)$. Summing over $m$ and $I$, we obtain
\[
s\ww{\nabla\zeta}+s\ww{\partial_t\zeta}+s^3\ww{\zeta}\le C\ww{\PP_c\zeta}+C\ww{\partial_t\zeta}+C\ww{\nabla\zeta}+Cs\ww{\zeta},
\]
and the last three terms are absorbed for $s$ large.
\end{proof}

The second estimate concerns sources in divergence form. It is the Riemannian version of the Carleman estimate in negative norms of Imanuvilov, Isakov and Yamamoto \cite[Theorem~3.2]{IIY03}, which goes back to Ruiz \cite{Ruiz92} and to Imanuvilov, see the references in \cite{IIY03}, for the wave equation.

\begin{lemma}\label{lem:carleman-neg}
Assume \Lh{} and \Gh{}. Let $k\in\{0,1,2\}$ and $c\in\{c_S,c_P\}$, and let $\mathcal D_1,\dots,\mathcal D_N$ be first-order differential operators in $x$ with smooth coefficients, $\mathcal D_m$ acting on $k_m$-forms with values in $k$-forms. There exist $s_0>0$ and $C>0$ with the following property, for $s\ge s_0$. If $\zeta\in L^2((0,T)\times M)$ is a $k$-form with compact support in $(0,T)\times U$, and
\[
\partial_t^2\zeta+c^2\DHo\zeta=\sum_{m=1}^N\mathcal D_mG_m+G_0\qquad\text{in } \mathcal D'((0,T)\times\inte M),
\]
where $G_m\in L^2((0,T)\times M)$ is a $k_m$-form and $G_0\in L^2((0,T)\times M)$ is a $k$-form, all with compact support in $(0,T)\times U$, then
\[
s\ww{\zeta}\le C\sum_{m=1}^N\ww{G_m}+Cs^{-2}\ww{G_0}.
\]
\end{lemma}

\begin{proof}
We use the charts and the partition of unity of the proof of Lemma~\ref{lem:carleman-forms}, and we write $x=(x^1,\dots,x^n)$ for the coordinates in $W_\ell''$ and $x^0=t$. Let $u_{\ell,I}=\chi_\ell\zeta_I$. In the coordinates of $W_\ell''$, the operator $P_0=\partial_t^2-c^2g^{jl}\partial_j\partial_l$ has smooth coefficients, and
\[
P_0=\partial_t^2-c^2\Delta_g+P_1,\qquad P_1=-c^2g^{jl}\Gamma^k_{jl}\partial_k ,
\]
where $\Gamma^k_{jl}$ are the Christoffel symbols of $g$. By \eqref{eq:DH-coord}, the expression of $\Delta_g$ in coordinates, and the identity $b\,\partial_jv=\partial_j(b\,v)-(\partial_jb)v$, every first-order term in $\zeta$ or in $G_m$ with smooth coefficients can be written as $\sum_j\partial_j(\cdot)$ plus a zero-order term. Therefore
\[
P_0u_{\ell,I}=h_0+\sum_{j=1}^n\partial_jg_j\qquad\text{in }\mathcal D',
\]
where $|g_j|\le C(\sum_m|G_m|+|\zeta|)$ for $1\le j\le n$ and $|h_0|\le C(|G_0|+\sum_m|G_m|+|\zeta|)$. All these functions have compact support in $(0,T)\times(W_\ell\cap U)$.

Let
\begin{equation}\label{eq:QQstar}
Q=(0,T)\times x(W_\ell\cap U),\qquad Q^*=I_T\times x(W_\ell'\cap U_1),
\end{equation}
so that $\overline Q\subset Q^*$ and the coefficients of $P_0$ are smooth on $\overline{Q^*}$. We claim that there exist $s_0>0$ and $C>0$ such that, for every function $v\in L^2(\R^{n+1})$ with compact support in $Q$ satisfying $P_0v=h_0+\sum_{j=0}^n\partial_jg_j$ in $\mathcal D'$ with $h_0,g_0,\dots,g_n\in L^2$, where $\partial_0=\partial_t$,
\begin{equation}\label{eq:IIY}
s\int e^{2s\phi}|v|^2\,\dd x\,\dd t\le C\int e^{2s\phi}\Bigl(s^{-2}|h_0|^2+\sum_{j=0}^n|g_j|^2\Bigr)\dd x\,\dd t,\qquad s\ge s_0 .
\end{equation}
This is the Carleman estimate in negative norms of Imanuvilov, Isakov and Yamamoto \cite[Theorem~3.2]{IIY03}, whose proof we reproduce in Appendix~\ref{app:IIY} in the present setting. The only ingredient of that proof which depends on the geometry is the $H^1$ Carleman estimate for the conjugated operator, and in Appendix~\ref{app:IIY} it is supplied by Lemma~\ref{lem:carleman-scalar}.

Applying \eqref{eq:IIY} to $u_{\ell,I}$, comparing the coordinate measure with $\dd v$ on $\supp\chi_\ell$, and summing over $\ell$ and $I$, we obtain
\[
s\ww{\zeta}\le C\sum_{m=1}^N\ww{G_m}+Cs^{-2}\ww{G_0}+C\ww{\zeta},
\]
because $|\zeta|^2\le C\sum_{\ell,I}|u_{\ell,I}|^2$. The last term is absorbed for $s$ large.
\end{proof}

%=====================================================================
\section{Hodge--Carleman estimates}\label{sec:HC}
%=====================================================================

Recall that $\PP=\partial_t^2-\varrho^{-1}\Lam$. For a vector field $z$ we write
\[
|\partial^1z|^2=|\partial_tz|^2+|\nabla z|^2,\qquad |\partial^2z|^2=|\partial_t^2z|^2+|\nabla\partial_tz|^2+|\nabla^2z|^2 .
\]
For $z$ with compact support in $(0,T)\times\inte M$ we set $F=\PP z$, $y=\delta z$ and $w=\dd z$. By \eqref{eq:P-hodge}, \eqref{eq:hodge-P} and \eqref{eq:hodge-S},
\begin{align}
\bigl(\partial_t^2+c_S^2\DHo\bigr)z+\mathcal Qz&=F-(c_P^2-c_S^2)\,\dd y,\label{eq:Z}\\
\bigl(\partial_t^2+c_P^2\DHo\bigr)y-\bigl\langle\dd c_P^2+\varrho^{-1}\dd(\lambda+2\mu),\dd y\bigr\rangle&=\delta F+\bigl\langle\dd c_S^2+\varrho^{-1}\dd\mu,\delta w\bigr\rangle-\mathcal M_0z,\label{eq:W}\\
\bigl(\partial_t^2+c_S^2\DHo\bigr)w+\dd c_S^2\wedge\delta w-\varrho^{-1}\nabla_{\nabla\mu}w&=\dd F-\bigl(\dd c_P^2-\varrho^{-1}\dd\lambda\bigr)\wedge\dd y-\mathcal M_2z.\label{eq:Psi}
\end{align}

\begin{theorem}\label{thm:HC2}
Assume \Lh{} and \Gh{}, let $T>0$, $\beta\in(0,\rho)$, and let $\phi$ be given by \eqref{eq:weight}. There exist $s_0>0$ and $C>0$ such that
\begin{equation}\label{eq:HC2}
s^4\ww{z}+s^2\ww{\partial^1z}+\ww{\partial^2z}\le C\bigl(s\ww{\PP z}+\ww{\nabla\PP z}\bigr),\qquad s\ge s_0,
\end{equation}
for every vector field $z\in C^\infty_c((0,T)\times U;TM)$.
\end{theorem}

\begin{proof}
All the unknowns are smooth and have compact support in $(0,T)\times U$.

\emph{Step 1.} We apply Lemma~\ref{lem:carleman-forms} to \eqref{eq:W} ($k=0$, $c=c_P$), to \eqref{eq:Psi} ($k=2$, $c=c_S$) and to \eqref{eq:Z} ($k=1$, $c=c_S$), the operator $\mathcal T$ being the first-order operator acting on the unknown in the left-hand side. Since
\[
|\delta F|+|\dd F|\le C|\nabla F|,\qquad |\mathcal M_0z|+|\mathcal M_2z|\le C(|\nabla z|+|z|),\qquad |\delta w|\le C|\nabla w|,\qquad |\dd y|=|\nabla y|,
\]
we obtain
\begin{align*}
s\ww{\partial^1y}+s^3\ww{y}&\le C\bigl(\ww{\nabla F}+\ww{\nabla w}+\ww{\nabla z}+\ww{z}\bigr),\\
s\ww{\partial^1w}+s^3\ww{w}&\le C\bigl(\ww{\nabla F}+\ww{\nabla y}+\ww{\nabla z}+\ww{z}\bigr),\\
s\ww{\partial^1z}+s^3\ww{z}&\le C\bigl(\ww{F}+\ww{\nabla y}\bigr).
\end{align*}
Adding the first two inequalities and absorbing $C\ww{\nabla y}+C\ww{\nabla w}$ for $s$ large, we find
\begin{equation}\label{eq:st1a}
s\ww{\partial^1y}+s^3\ww{y}+s\ww{\partial^1w}+s^3\ww{w}\le C\bigl(\ww{\nabla F}+\ww{\nabla z}+\ww{z}\bigr).
\end{equation}
Inserting \eqref{eq:st1a} into the third inequality,
\[
s\ww{\partial^1z}+s^3\ww{z}\le C\ww{F}+Cs^{-1}\bigl(\ww{\nabla F}+\ww{\nabla z}+\ww{z}\bigr),
\]
and absorbing the last terms we find, for $s$ large,
\begin{equation}\label{eq:st1c}
s\ww{\partial^1z}+s^3\ww{z}\le C\ww{F}+Cs^{-1}\ww{\nabla F}.
\end{equation}
Using \eqref{eq:st1c} in \eqref{eq:st1a} we also obtain
\begin{equation}\label{eq:st1d}
s\ww{\partial^1y}+s^3\ww{y}+s\ww{\partial^1w}+s^3\ww{w}\le C\ww{\nabla F}+Cs^{-1}\ww{F}.
\end{equation}

\emph{Step 2 (second spatial derivatives).} For fixed $t$, let $\zeta=e^{s\phi}z$, which has compact support in $U$. We have
\[
\dd\zeta=e^{s\phi}\bigl(w+s\,\dd\phi\wedge z\bigr),\qquad \delta\zeta=e^{s\phi}\bigl(y-s\langle\nabla\phi,z\rangle\bigr).
\]
Since $\phi$ and its spatial derivatives up to order two are bounded on $[0,T]\times\overline U$, we obtain, for $s\ge1$,
\[
|\nabla\dd\zeta|+|\nabla\delta\zeta|\le Ce^{s\phi}\bigl(|\nabla w|+|\nabla y|+s|w|+s|y|+s|\nabla z|+s^2|z|\bigr),\qquad |\nabla\zeta|\le Ce^{s\phi}\bigl(|\nabla z|+s|z|\bigr).
\]
Moreover
\[
e^{s\phi}|\nabla^2z|\le|\nabla^2\zeta|+Cs|\nabla\zeta|+Cs^2|\zeta| .
\]
Applying \eqref{eq:gaffney2} to $\zeta$ and integrating in $t$, we get
\[
\ww{\nabla^2z}\le C\bigl(\ww{\nabla w}+\ww{\nabla y}+s^2\ww{w}+s^2\ww{y}+s^2\ww{\nabla z}+s^4\ww{z}\bigr),
\]
and \eqref{eq:st1c}, \eqref{eq:st1d} yield
\begin{equation}\label{eq:st2}
\ww{\nabla^2z}\le C\bigl(s\ww{F}+\ww{\nabla F}\bigr).
\end{equation}

\emph{Step 3 (derivatives involving time).} Let $\zeta_1=e^{s\phi}\partial_tz$. As in Step 2,
\[
\dd\zeta_1=e^{s\phi}\bigl(\partial_tw+s\,\dd\phi\wedge\partial_tz\bigr),\qquad \delta\zeta_1=e^{s\phi}\bigl(\partial_ty-s\langle\nabla\phi,\partial_tz\rangle\bigr).
\]
Applying \eqref{eq:gaffney1} to $\zeta_1$ and using \eqref{eq:st1c} and \eqref{eq:st1d},
\[
\ww{\nabla\partial_tz}\le C\bigl(\ww{\partial_tw}+\ww{\partial_ty}+s^2\ww{\partial_tz}\bigr)\le C\bigl(s\ww{F}+\ww{\nabla F}\bigr).
\]
Finally, $\partial_t^2z=\varrho^{-1}\Lam z+F$ and $|\Lam z|\le C(|\nabla^2z|+|\nabla z|+|z|)$, so that \eqref{eq:st1c} and \eqref{eq:st2} give $\ww{\partial_t^2z}\le C(s\ww{F}+\ww{\nabla F})$. Multiplying \eqref{eq:st1c} by $s$ and collecting the estimates, we obtain \eqref{eq:HC2}.
\end{proof}

\begin{theorem}\label{thm:HC0}
Assume \Lh{} and \Gh{}, let $T>0$, $\beta\in(0,\rho)$, and let $\phi$ be given by \eqref{eq:weight}. There exist $s_0>0$ and $C>0$ such that
\begin{equation}\label{eq:HC0}
s\ww{z}+s\ww{\delta z}+s\ww{\dd z}\le C\ww{\PP z},\qquad s\ge s_0,
\end{equation}
for every vector field $z\in H^1((0,T)\times M;TM)$ with compact support in $(0,T)\times U$ such that $\PP z\in L^2$.
\end{theorem}

\begin{proof}
Let $F=\PP z\in L^2$, $y=\delta z\in L^2$ and $w=\dd z\in L^2$. By Propositions~\ref{prop:three-forms} and~\ref{prop:hodge}, the identities \eqref{eq:Z}, \eqref{eq:W} and \eqref{eq:Psi} hold in the sense of distributions. Each term in these identities, except the second-order operators $\partial_t^2+c^2\DHo$, is of the form $\mathcal DG$ with a differential operator $\mathcal D$ of order at most one with smooth coefficients and $G\in\{F,y,w,z\}$. For instance, $(c_P^2-c_S^2)\,\dd y=\dd\bigl((c_P^2-c_S^2)y\bigr)-y\,\dd(c_P^2-c_S^2)$. Hence Lemma~\ref{lem:carleman-neg}, applied to \eqref{eq:Z} with $k=1$ and $c=c_S$, to \eqref{eq:W} with $k=0$ and $c=c_P$, and to \eqref{eq:Psi} with $k=2$ and $c=c_S$, gives
\[
s\ww{z}+s\ww{y}+s\ww{w}\le C\bigl(\ww{F}+\ww{z}+\ww{y}+\ww{w}\bigr),
\]
and \eqref{eq:HC0} follows for $s$ large.
\end{proof}

\begin{remark}\label{rem:loss}
For the scalar wave operator the Carleman estimate with an $L^2$ source controls $s^3\ww{u}+s\ww{\partial^1u}$, see Lemma~\ref{lem:carleman-scalar}. Theorem~\ref{thm:HC0} controls only $s\ww{z}$, without the gradient, and Theorem~\ref{thm:HC2} requires the spatial gradient of the source. Both losses come from the first-order coupling of $z$ with $\delta z$ and $\dd z$ in \eqref{eq:Z}--\eqref{eq:Psi}, which forces the use of the estimate in negative norms. Theorem~\ref{thm:HC0} is the Riemannian counterpart of \cite[Theorem~3.1]{IIY03} and of \cite[Theorem~2.3]{IY05}, which control the solution in $L^2$ only, and Theorem~\ref{thm:HC2} is the counterpart of \cite[Theorem~2.1]{IY05}, which requires the source in $H^1$. In two- and three-dimensional Euclidean domains, \cite[Theorem~2.2]{IY05} and \cite[Theorem~20.2]{IY05b} control $|\nabla z|^2+s^2|z|^2$ by an $L^2$ source, with the loss of one power of $s$ with respect to the scalar case, through a microlocal factorization at the boundary which has no analogue in the present setting. The pseudoconvexity hypothesis of \cite[Theorem~3.1]{IIY03}, with respect to the two operators $\varrho\partial_t^2-\mu\Delta$ and $\varrho\partial_t^2-(\lambda+2\mu)\Delta$, is the Euclidean counterpart of (G2).
\end{remark}

%=====================================================================
\section{Unique continuation with rough potentials}\label{sec:ucp}
%=====================================================================

\begin{theorem}\label{thm:ucp}
Assume \Lh{} and \Gh{}, and let $T>T_*$. Let $V\in L^\infty((0,T)\times M;\operatorname{End}(TM))$, and let
\[
z\in L^\infty(0,T;H^1_0(M;TM)),\qquad \partial_tz\in L^\infty(0,T;L^2(M;TM)),
\]
satisfy
\[
\varrho\,\partial_t^2z-\Lam z+Vz=0\quad\text{in }\mathcal D'((0,T)\times\inte M),\qquad z=0\quad\text{a.e. in }(0,T)\times\omega .
\]
Then $z=0$ in $(0,T)\times M$.
\end{theorem}

\begin{proof}
\emph{Step 0 (regularity and energy identity).} We first identify $z$ with a solution in the sense of Lions and Magenes \cite[Chapter~3, Section~8]{LM72}, with $\mathbb V=H^1_0(M;TM)$ and $H=L^2_\varrho$ in the roles of the spaces called $V$ and $H$ there, and the form $a$ of \eqref{eq:form-a}, which is symmetric and coercive by Lemma~\ref{lem:korn}. Testing the equation against $\theta(t)\varphi(x)$ with $\theta\in C^\infty_c(0,T)$ and $\varphi\in C^\infty_c(\inte M;TM)$, and using that $\Lam$ is formally self-adjoint, so that $\langle-\Lam z(t),\varphi\rangle=a(z(t),\varphi)$ by Lemma~\ref{lem:green}, we see that the equation in $\mathcal D'((0,T)\times\inte M)$ is equivalent to
\[
\frac{\dd^2}{\dd t^2}(\varrho z,\varphi)+a(z,\varphi)+(Vz,\varphi)=0\quad\text{in }\mathcal D'(0,T),\qquad\varphi\in H^1_0,
\]
where the passage from $\varphi\in C^\infty_c(\inte M;TM)$ to $\varphi\in H^1_0$ uses the density of the former space in the latter and the boundedness of $a$. This is the weak formulation of \cite[Chapter~3, Section~8]{LM72} for the equation $\partial_t^2z+Az=-\varrho^{-1}Vz$ in $H$, with $\varrho^{-1}Vz\in L^2(0,T;H)$. Since $z\in L^2(0,T;\mathbb V)$ and $\partial_tz\in L^2(0,T;H)$, the identity above shows that $\partial_t^2z\in L^2(0,T;\mathbb V')$, so that $z\in C([0,T];H)$, $\partial_tz\in C([0,T];\mathbb V')$, and $z(0)\in H$, $\partial_tz(0)\in \mathbb V'$ are well defined, as in \cite[Chapter~3, Remark~8.2]{LM72}. By \cite[Chapter~3, Lemma~8.1]{LM72}, applied with $X=\mathbb V$, $Y=H$ to $z$ and with $X=H$, $Y=\mathbb V'$ to $\partial_tz$, the functions $z$ and $\partial_tz$ are weakly continuous with values in $\mathbb V$ and $H$ respectively, so that $z(0)\in\mathbb V$ and $\partial_tz(0)\in H$. Thus $z$ belongs to the class \textup{(8.7)--(8.8)} of \cite[Chapter~3, Theorem~8.1]{LM72}, in which that theorem asserts uniqueness, and $z$ is therefore the solution given by that theorem with the data $(z(0),\partial_tz(0))\in\mathbb V\times H$. Hence $z\in C([0,T];H^1_0)\cap C^1([0,T];L^2)$ by \cite[Chapter~3, Theorem~8.2]{LM72}, and the energy equality of \cite[Chapter~3, Lemma~8.3]{LM72} gives, for $t,t'\in[0,T]$,
\begin{equation}\label{eq:energy-V}
E(t)=E(t')-\int_{t'}^t(Vz,\partial_tz)\,\dd\tau,\qquad E=\frac12\bigl(\|\partial_tz\|_\varrho^2+a(z,z)\bigr).
\end{equation}

\emph{Step 1 (choice of the parameters).} The parameters are chosen in the following order. The time $T>T_*$ is given; we choose $\beta$ below, then $\gamma_0$ and the weight $\phi$ are those of Section~\ref{sec:geometry} for this pair $(T,\beta)$, then $\tau_0$, $\delta_1$ and $\varepsilon_0$ are chosen, and finally $s$ is sent to infinity. Let $m_0=\min_{\overline U}d$ and $m_1=\max_{\overline U}d$. Since $T>T_*$, by \eqref{eq:Tstar} we may choose $\beta\in(0,\rho)$ such that $\beta T^2/4>m_1-m_0$, and we set
\[
\tau_0=\frac13\Bigl(\frac{\beta T^2}{4}-(m_1-m_0)\Bigr)>0 .
\]
Then we choose $\delta_1\in(0,T/2)$ and $\varepsilon_0\in(0,T/2-\delta_1)$ such that
\[
\beta\Bigl(\frac T2-\delta_1\Bigr)^2\ge\frac{\beta T^2}{4}-\tau_0,\qquad \beta\varepsilon_0^2\le\tau_0 .
\]
Then, for $x\in\overline U$,
\[
\psi(t,x)\le m_0-2\tau_0+\beta_0\quad\text{if } t\in[0,\delta_1]\cup[T-\delta_1,T],\qquad \psi(t,x)\ge m_0-\tau_0+\beta_0\quad\text{if } |t-T/2|\le\varepsilon_0 .
\]
We set $\phi_-=e^{\gamma_0(m_0-2\tau_0+\beta_0)}$, $\phi_+=e^{\gamma_0(m_0-\tau_0+\beta_0)}$, and $J_0=[T/2-\varepsilon_0,T/2+\varepsilon_0]$.

\emph{Step 2 (cut-off).} Let $\eta\in C^\infty_c((0,T))$ with $0\le\eta\le1$ and $\eta=1$ on $[\delta_1,T-\delta_1]$, and let $\chi\in C^\infty_c(U)$ with $0\le\chi\le1$ and $\chi=1$ on an open neighbourhood $\cN$ of $K$. Then $\nabla\chi$ vanishes on $\cN$ and outside $U$, so that $\supp\nabla\chi\subset U\setminus\cN\subset M\setminus K=\omega$. The vector field $v=\eta\chi z$ belongs to $H^1((0,T)\times M;TM)$, because $z\in L^\infty(0,T;H^1_0)$ and $\partial_tz\in L^\infty(0,T;L^2)$, it has compact support in $(0,T)\times U$, and
\[
\PP v=\eta\chi\,\PP z+\chi\bigl(\eta''z+2\eta'\partial_tz\bigr)-\eta\varrho^{-1}[\Lam,\chi]z .
\]
The commutator $[\Lam,\chi]z$ is a first-order expression in $z$ with smooth coefficients supported in $\supp\nabla\chi\subset\omega$. Since $z=0$ a.e. in $(0,T)\times\omega$, also $\nabla z=0$ a.e. there, because the gradient of a Sobolev function vanishes almost everywhere on its zero set, and hence $[\Lam,\chi]z=0$. Using $\PP z=-\varrho^{-1}Vz$, we obtain
\[
\PP v=-\varrho^{-1}Vv+\chi\bigl(\eta''z+2\eta'\partial_tz\bigr)\in L^2 .
\]

\emph{Step 3 (Carleman estimate).} By Theorem~\ref{thm:HC0},
\[
s\ww{v}\le 2C\|\varrho^{-1}V\|_{L^\infty}^2\ww{v}+2C\int_0^T\!\!\int_Me^{2s\phi}\chi^2|\eta''z+2\eta'\partial_tz|^2\,\dd v\,\dd t .
\]
For $s\ge\max\{s_0,4C\|\varrho^{-1}V\|_{L^\infty}^2\}$ the first term is absorbed. Since $\eta'$ and $\eta''$ vanish outside $[0,\delta_1]\cup[T-\delta_1,T]$, where $\phi\le\phi_-$, and since $v=z$ on $J_0\times K$, where $\phi\ge\phi_+$, we arrive at
\[
\int_{J_0}\!\int_K|z|^2\,\dd v\,\dd t\le Cs^{-1}e^{-2s(\phi_+-\phi_-)}\int_0^T\bigl(\|z\|^2+\|\partial_tz\|^2\bigr)\dd t .
\]
Letting $s\to\infty$ we conclude that $z=0$ a.e. in $J_0\times K$. Since $z=0$ a.e. in $J_0\times\omega$ and $M=K\cup\omega$, we obtain $z(t)=0$ for $t\in J_0$, and then $\partial_tz(t)=0$ for $t\in J_0$.

\emph{Step 4 (conclusion).} We use \eqref{eq:energy-V} with $t'=T/2$ and the inequality
\[
\|z\|\le C\|z\|_\varrho\le C\lambda_1^{-1/2}a(z,z)^{1/2}.
\]
Then
\[
E(t)\le\Bigl|\int_{T/2}^t\|V\|_{L^\infty}\|z\|\|\partial_tz\|\,\dd\tau\Bigr|\le C\|V\|_{L^\infty}\Bigl|\int_{T/2}^tE(\tau)\,\dd\tau\Bigr| .
\]
The Gronwall lemma gives $E\equiv0$ on $[0,T]$, hence $z=0$.
\end{proof}

\begin{corollary}\label{cor:eigen}
Assume \Lh{} and \Gh{}. If $\varphi\in D(A)$ satisfies $A\varphi=\Lambda\varphi$ for some $\Lambda\in\R$ and $\varphi=0$ a.e. in $\omega$, then $\varphi=0$.
\end{corollary}

\begin{proof}
We may assume that $\varphi\ne0$, so that $\Lambda>0$ because $A$ is positive. Apply Theorem~\ref{thm:ucp} with $T>T_*$, $V=0$ and $z(t,x)=\cos(\sqrt\Lambda\,t)\varphi(x)$.
\end{proof}

\begin{remark}\label{rem:ucp}
In Theorem~\ref{thm:ucp} the potential is only bounded and measurable in $(t,x)$, and the coefficients $\varrho$, $\lambda$ and $\mu$ are arbitrary smooth functions satisfying \Lh{}. The roughness of the potential is not a feature of the present method but of every Carleman-based uniqueness theorem; in the Euclidean case, the proof of \cite[Theorem~3.1]{IIY03} already handles zero-order terms with bounded measurable coefficients, which are absorbed by the Carleman estimate exactly as here. The comparison with the sharp uniqueness theorems of \cite{Tat95,RZ98,Hor97,LL19} is one of hypotheses rather than of strength. Those theorems do not require any convexity but need coefficients analytic in time, which is the reason why the analyticity of the nonlinearity is assumed in \cite{JL13}, whereas Theorem~\ref{thm:ucp} requires no analyticity but the convexity hypothesis \Gh{}.
\end{remark}

%=====================================================================
\section{Internal observability}\label{sec:obs}
%=====================================================================

In this section we consider the linear undamped problem
\begin{equation}\label{eq:linear}
\varrho\,\partial_t^2u-\Lam u=0\ \ \text{in }(0,T)\times M,\qquad u=0\ \ \text{on }(0,T)\times\partial M,
\end{equation}
whose solutions with data in $\cH$ belong to $C([0,T];H^1_0)\cap C^1([0,T];L^2)$ and conserve the energy $E_u(t)=\frac12(\|\partial_tu\|_\varrho^2+a(u,u))$. For inhomogeneous and anisotropic elastic waves in Euclidean domains, with $\omega$ a neighbourhood of the whole boundary as in (G1), internal observability inequalities were obtained in \cite{MNY21} by multipliers of Morawetz type, under a one-sided radial condition on the coefficients and with no convexity requirement. The result below concerns the Riemannian setting under \Gh{}, and its proof passes through the second energy level, because the Carleman estimate of Theorem~\ref{thm:HC2} requires the source in $H^1$.

\begin{lemma}\label{lem:local-energy}
Assume \Lh{}. Let $\theta\in C^\infty(M)$ with $0\le\theta\le1$ and $\supp\theta\subset\omega$, and let $\vartheta\in C^\infty_c((0,T))$ with $0\le\vartheta\le1$. There exists $C>0$ such that every solution of \eqref{eq:linear} with data in $\cH$ satisfies
\[
\int_0^T\!\vartheta\,\|\theta\nabla u\|^2\,\dd t\le C\int_0^T\!\!\int_\omega\bigl(|\partial_tu|^2+|u|^2\bigr)\dd v\,\dd t .
\]
\end{lemma}

\begin{proof}
Both sides are continuous on $\cH$, so we may assume that $u$ is a strong solution. Since $\theta^2u(t)\in H^1_0$, Lemma~\ref{lem:green} and an integration by parts in $t$ give
\[
\int_0^T\vartheta\,a(u,\theta^2u)\,\dd t=\int_0^T\vartheta\|\theta\partial_tu\|_\varrho^2\,\dd t+\int_0^T\vartheta'(\varrho\,\partial_tu,\theta^2u)\,\dd t .
\]
Writing $\mathcal E(u)=\lambda(\dv u)^2+2\mu|\Def u|^2$, we have
\[
a(u,\theta^2u)=\int_M\theta^2\mathcal E(u)\,\dd v+\mathcal R,\qquad |\mathcal R|\le C\int_M\theta|\nabla\theta||\nabla u||u|\,\dd v,
\]
and $\mathcal E(u)\ge c_0|\Def u|^2$ by \eqref{eq:pointwise}. Applying Lemma~\ref{lem:korn}(ii) to $\theta u\in H^1_0$,
\[
\|\theta\nabla u\|^2\le2\|\nabla(\theta u)\|^2+2\||\nabla\theta|\,u\|^2\le4C_K\|\theta\Def u\|^2+C\|u\|^2_{L^2(\supp\theta)} .
\]
Collecting these relations and using the Young inequality
\[
C\theta|\nabla\theta||\nabla u||u|\le\varepsilon\theta^2|\nabla u|^2+C_\varepsilon|\nabla\theta|^2|u|^2
\]
with $\varepsilon$ small, we obtain the claim, because $\supp\theta\subset\omega$.
\end{proof}

\begin{theorem}\label{thm:obs}
Assume \Lh{} and \Gh{}, and let $T>T_*$. There exists $C_T>0$, which depends on $T$ and is not estimated, such that every solution of \eqref{eq:linear} with data in $\cH$ satisfies
\begin{equation}\label{eq:obs}
E_u(0)\le C_T\int_0^T\!\!\int_\omega|\partial_tu|^2\,\dd v\,\dd t .
\end{equation}
\end{theorem}

\begin{proof}
\emph{Step 1 (an estimate at the second energy level).} Let $u$ be a solution of \eqref{eq:linear} with data in $\bigcap_kD(A^{k/2})$. By Lemma~\ref{lem:operatorA} and the regularity statement recalled in its proof, applied inductively, $D(A^{k/2})\subset H^k$ for every $k$, and the inclusion is continuous by the closed graph theorem. For even $k=2m$ this follows by descending induction from $A^{m-1}u\in D(A)$ and $-\Lam A^{m-1}u=\varrho A^mu\in L^2$, because $-\Lam A^{j-1}u=\varrho A^ju\in H^{2(m-j)}$ gives $A^{j-1}u\in H^{2(m-j)+2}$ for $j=m,\dots,1$, and for odd $k=2m+1$ in the same way from $A^mu\in D(A^{1/2})=H^1_0$ and $-\Lam A^{m-1}u=\varrho A^mu\in H^1$. Hence $u\in C^j([0,T];H^k)$ for all $j$ and $k$, so that $u\in C^\infty([0,T]\times M)$ by the Sobolev embedding, and $\partial_tu$ is also a solution of \eqref{eq:linear}. We denote by $E_1=E_{\partial_tu}$ its energy, which is constant in time. Since $\partial_t^2u=-Au$, the spectral theorem gives
\begin{equation}\label{eq:E1-E}
E_1=\frac12\bigl(\|Au\|_\varrho^2+a(\partial_tu,\partial_tu)\bigr)\ge\lambda_1E_u .
\end{equation}
Let $\beta,\tau_0,\delta_1,\varepsilon_0,\phi_\pm,J_0$ be as in Step~1 of the proof of Theorem~\ref{thm:ucp}. Let $\cN_0\subset U$ be an open neighbourhood of $K$ and $\chi\in C^\infty_c(U)$ with $\chi=1$ on $\cN_0$, so that $\omega_1=\supp\nabla\chi$ is a compact subset of $\omega\cap\inte M$. Let $\eta$ be as in the proof of Theorem~\ref{thm:ucp}. The field $v=\eta\chi u$ is smooth with compact support in $(0,T)\times U$, and
\[
\PP v=\chi\bigl(\eta''u+2\eta'\partial_tu\bigr)-\eta\varrho^{-1}[\Lam,\chi]u .
\]
We estimate the right-hand side of \eqref{eq:HC2}. The commutator term is supported in $\supp\eta\times\omega_1$, where
\[
|[\Lam,\chi]u|+|\nabla[\Lam,\chi]u|\le C\bigl(|\nabla^2u|+|\nabla u|+|u|\bigr).
\]
The terms with $\eta'$ and $\eta''$ are supported in $([0,\delta_1]\cup[T-\delta_1,T])\times\overline U$, where $\phi\le\phi_-$, and are bounded with their gradients by $C(|u|+|\nabla u|+|\partial_tu|+|\nabla\partial_tu|)$. Using Lemma~\ref{lem:korn} and the conservation of $E_u$ and $E_1$, we obtain, with $\phi_{\max}=\max_{[0,T]\times\overline U}\phi$,
\[
s\ww{\PP v}+\ww{\nabla\PP v}\le Cse^{2s\phi_{\max}}\int_{\supp\eta}\int_{\omega_1}\bigl(|\nabla^2u|^2+|\nabla u|^2+|u|^2\bigr)+CsTe^{2s\phi_-}\bigl(E_u+E_1\bigr).
\]
Let $\theta_1,\theta_2\in C^\infty_c(\omega\cap\inte M)$ with $0\le\theta_i\le1$, $\theta_1=1$ near $\omega_1$ and $\theta_2=1$ near $\supp\theta_1$, and let $\vartheta\in C^\infty_c((0,T))$ with $\vartheta=1$ on $\supp\eta$. By Lemma~\ref{lem:interior-elliptic} applied to $\theta_1u(t)$, and since $\Lam u=\varrho\,\partial_t^2u$,
\[
\|\nabla^2u\|_{L^2(\omega_1)}\le C\bigl(\|\partial_t^2u\|_{L^2(\omega)}+\|\theta_2\nabla u\|+\|u\|_{L^2(\omega)}\bigr).
\]
By Lemma~\ref{lem:local-energy} with $\theta_2$ and $\vartheta$, the integral over $\supp\eta\times\omega_1$ above is bounded by $C\cO(u)$, where
\[
\cO(u)=\int_0^T\!\!\int_\omega\bigl(|\partial_t^2u|^2+|\partial_tu|^2+|u|^2\bigr)\dd v\,\dd t .
\]
On the other hand, $v=u$ on $J_0\times\cN_0$, where $\phi\ge\phi_+$, so that the left-hand side of \eqref{eq:HC2} is bounded from below by
\[
e^{2s\phi_+}\int_{J_0}\int_{\cN_0}\bigl(|\partial_t^2u|^2+|\nabla\partial_tu|^2\bigr)\dd v\,\dd t .
\]
Let $\theta\in C^\infty(M)$ with $0\le\theta\le1$, $\theta=1$ on $M\setminus\cN_0$ and $\supp\theta\subset\omega$, which exists because $M\setminus\cN_0$ is a compact subset of $\omega$. Lemma~\ref{lem:local-energy} applied to the solution $\partial_tu$, with $\theta$ and with $\vartheta_0\in C^\infty_c((0,T))$ equal to $1$ on $J_0$, gives
\[
\int_{J_0}\!\int_{M\setminus\cN_0}\bigl(|\partial_t^2u|^2+|\nabla\partial_tu|^2\bigr)\dd v\,\dd t\le C\cO(u).
\]
We add the two lower bounds and use the inequality
\[
\|\partial_t^2u\|_\varrho^2+a(\partial_tu,\partial_tu)\le C\bigl(\|\partial_t^2u\|^2+\|\nabla\partial_tu\|^2\bigr).
\]
We obtain
\[
4\varepsilon_0E_1=\int_{J_0}\bigl(\|\partial_t^2u\|_\varrho^2+a(\partial_tu,\partial_tu)\bigr)\dd t\le C\bigl(1+se^{2s(\phi_{\max}-\phi_+)}\bigr)\cO(u)+Cse^{-2s(\phi_+-\phi_-)}\bigl(E_u+E_1\bigr).
\]
By \eqref{eq:E1-E}, $E_u+E_1\le(1+\lambda_1^{-1})E_1$, and fixing $s$ large we conclude that
\begin{equation}\label{eq:obs-E1}
E_1\le C\cO(u).
\end{equation}

\emph{Step 2 (a weak observability inequality).} Let $\tilde u$ be a solution of \eqref{eq:linear} with data in $\bigcap_kD(A^{k/2})$, and let $u$ be the solution with $u(0)=-A^{-1}\partial_t\tilde u(0)$ and $\partial_tu(0)=\tilde u(0)$. Then $\partial_tu$ and $\tilde u$ solve \eqref{eq:linear} with the same data, so that $\partial_tu=\tilde u$, and $u=-A^{-1}\partial_t^2u=-A^{-1}\partial_t\tilde u$. Since $E_1=E_{\tilde u}$, \eqref{eq:obs-E1} gives
\begin{equation}\label{eq:weak-obs}
E_{\tilde u}(0)\le C\int_0^T\!\!\int_\omega|\partial_t\tilde u|^2\,\dd v\,\dd t+C\int_0^T\bigl(\|\tilde u\|^2+\|A^{-1}\partial_t\tilde u\|^2\bigr)\dd t .
\end{equation}
All the terms in \eqref{eq:weak-obs} are continuous with respect to the data in $\cH$, and therefore \eqref{eq:weak-obs} holds for all solutions with data in $\cH$. The argument applies to every $\tau\in(T_*,T]$ in place of $T$, with a constant depending on $\tau$.

\emph{Step 3 (compactness).} We denote by $H^{-1}$ the dual of $H^1_0(M;TM)$ with respect to the pivot space $L^2_\varrho$, and we still write $A$ for the extension of $A$ to an isomorphism of $H^1_0$ onto $H^{-1}$ defined by $\langle Au,v\rangle=a(u,v)$, so that $A^{-1}$ is bounded from $H^{-1}$ into $H^1_0$ and coincides with the inverse of $A$ on $L^2_\varrho$. Let $\mathcal K_\tau(\tilde u_0,\tilde u_1)=(\tilde u,A^{-1}\partial_t\tilde u)$, where $\tilde u$ is the solution with data $(\tilde u_0,\tilde u_1)\in\cH$ on $[0,\tau]$. If $(\tilde u_0,\tilde u_1)$ ranges over a bounded subset of $\cH$, then $\tilde u$ is bounded in $C([0,\tau];H^1_0)$ and $\partial_t\tilde u$ is bounded in $C([0,\tau];L^2)$, while $\partial_t^2\tilde u=-A\tilde u$ is bounded in $C([0,\tau];H^{-1})$. By the Aubin--Dubinskii--Lions theorem \cite[Theorem~A.3.7]{Ch15}, $\tilde u$ and $\partial_t\tilde u$ are relatively compact in $C([0,\tau];L^2)$ and $C([0,\tau];H^{-1})$, respectively. Hence $\mathcal K_\tau$ is a compact linear operator from $\cH$ into $L^2(0,\tau;L^2)^2$.

\emph{Step 4 (the space of invisible solutions).} For $\tau\in(T_*,T]$, let $\mathscr V_\tau$ be the closed subspace of $\cH$ formed by the data whose solutions satisfy $\partial_t\tilde u=0$ a.e. in $(0,\tau)\times\omega$. By \eqref{eq:weak-obs}, $E_{\tilde u}(0)\le C\|\mathcal K_\tau(\tilde u_0,\tilde u_1)\|^2$ on $\mathscr V_\tau$, so that the unit ball of $\mathscr V_\tau$ is compact and $\dim\mathscr V_\tau<\infty$. Moreover $\mathscr V_\tau\supset\mathscr V_{\tau'}$ for $\tau<\tau'$. Fix $T'\in(T_*,T)$. The function $\tau\mapsto\dim\mathscr V_\tau$ is nonincreasing, integer valued and bounded by $\dim\mathscr V_{T'}$ on $[T',T]$, so there exist $\tau_1<\tau_2$ in $[T',T]$ such that $\mathscr V_{\tau_1}=\mathscr V_{\tau_2}$. We denote this space by $\mathscr V$. If $(\tilde u_0,\tilde u_1)\in\mathscr V$ and $0<\varepsilon<\tau_2-\tau_1$, the data at time $\varepsilon$ belong to $\mathscr V_{\tau_2-\varepsilon}\subset\mathscr V_{\tau_1}=\mathscr V$. Let $(S(t))_{t\in\R}$ be the unitary group generated by $\cA(u,v)=(v,-Au)$ on $\cH$, with $D(\cA)=D(A)\times H^1_0$, whose orbits are the solutions of \eqref{eq:linear}. We have shown that $S(\varepsilon)\mathscr V\subset\mathscr V$ for $0<\varepsilon<\tau_2-\tau_1$, hence $S(t)\mathscr V\subset\mathscr V$ for all $t\ge0$ by the group property, and $S(t)\mathscr V=\mathscr V$ for all $t\ge0$ because $S(t)$ is injective and $\mathscr V$ is finite dimensional, so that $\mathscr V$ is invariant under $S(t)$ for all $t\in\R$. The restriction of $(S(t))$ to $\mathscr V$ is a strongly continuous group on a finite-dimensional space, hence generated by a bounded operator, which is the restriction of $\cA$. Thus $\mathscr V\subset D(\cA)$ and $\cA\mathscr V\subset\mathscr V$. If $\mathscr V\ne\{0\}$, the complexification $\mathscr V_\C=\mathscr V+i\mathscr V$ is invariant under $\cA$, and $\cA|_{\mathscr V_\C}$ has an eigenvector $(\varphi,\varkappa\varphi)\neq0$ with eigenvalue $\varkappa\in\C$; then $A\varphi=-\varkappa^2\varphi$ and $\varphi\ne0$. Since $A$ is positive, $\varkappa^2<0$, so that $\varkappa\ne0$, and the solution $e^{\varkappa t}\varphi$ of the complexified problem, whose data belong to $\mathscr V_\C$, satisfies $\varkappa e^{\varkappa t}\varphi=0$ in $\omega$. Hence $\varphi=0$ a.e. in $\omega$, the real and imaginary parts of $\varphi$ are eigenfunctions of $A$ vanishing in $\omega$, and Corollary~\ref{cor:eigen} gives $\varphi=0$, a contradiction. Therefore $\mathscr V_T\subset\mathscr V=\{0\}$.

\emph{Step 5 (conclusion).} If \eqref{eq:obs} fails, there are solutions $\tilde u^k$ with $E_{\tilde u^k}(0)=1$ and $\int_0^T\int_\omega|\partial_t\tilde u^k|^2\to0$. Passing to a subsequence, the data converge weakly in $\cH$ to some $(\tilde u_0,\tilde u_1)$, and by Step~3, $\mathcal K_T$ of the data converges strongly. The corresponding solution $\tilde u$ satisfies $\partial_t\tilde u=0$ a.e. in $(0,T)\times\omega$ by the weak lower semicontinuity of the norm, so that $(\tilde u_0,\tilde u_1)\in\mathscr V_T=\{0\}$. On the other hand, \eqref{eq:weak-obs} applied to $\tilde u^k$ gives
\[
1\le C\int_0^T\!\!\int_\omega|\partial_t\tilde u^k|^2\,\dd v\,\dd t+C\|\mathcal K_T(\tilde u^k_0,\tilde u^k_1)\|^2,
\]
and in the limit $\|\mathcal K_T(\tilde u_0,\tilde u_1)\|^2\ge C^{-1}>0$, which is impossible.
\end{proof}

%=====================================================================
\section{Long-time dynamics}\label{sec:dynamics}
%=====================================================================

\begin{theorem}\label{thm:main}
Assume \Lh{}, (H1)--(H4), so that in particular $n=3$, and \Gh{}. Then the dynamical system $(\cH,S_t)$ generated by \eqref{eq:main} has a compact global attractor $\mathfrak A$ with the following properties.
\begin{enumerate}[label=(\roman*)]
\item $\mathfrak A=\mathbb M^u(\cN)$, the unstable manifold emanating from the set $\cN$ of equilibria in the sense of \cite[Definition~2.3.10]{Ch15}, namely the set of points through which there passes a complete trajectory converging to $\cN$ as $t\to-\infty$.
\item $\mathfrak A$ has finite fractal dimension in $\cH$.
\item $\mathfrak A$ is a bounded subset of $(H^2\cap H^1_0)\times H^1_0$. More precisely, there exists $R>0$ such that every full trajectory $\{(u(t),\partial_tu(t)):t\in\R\}\subset\mathfrak A$ satisfies $\partial_t^2u\in L^\infty(\R;L^2)$ and
\[
\|u(t)\|_{H^2}+\|\partial_tu(t)\|_{H^1}+\|\partial_t^2u\|_{L^\infty(\R;L^2)}\le R,\qquad t\in\R .
\]
\item $(\cH,S_t)$ possesses a generalized fractal exponential attractor whose dimension is finite in $L^2(M;TM)\times H^{-1}(M;TM)$.
\item For every $y\in\cH$, the $\omega$-limit set of $y$ is contained in $\cN$. In particular, if $\cN$ is finite, every trajectory converges to an equilibrium.
\end{enumerate}
\end{theorem}

We do not prove that $\Phi$ is a strict Lyapunov function on $\cH$, that is, that $(\cH,S_t)$ is a gradient system in the sense of \cite[Definition~2.4.1]{Ch15}; the characterization (i) is obtained instead from the rigidity of the complete trajectories with constant energy, Lemma~\ref{lem:rigidity} below. The proof rests on the next three results. The first one is a stabilizability estimate in the sense of Chueshov and Lasiecka \cite{CL08,CL10,Ch15}.

\begin{lemma}\label{lem:stab}
Assume \Lh{}, (H1)--(H4) and \Gh{}. Let $B\subset\cH$ be a bounded forward invariant set. There exist $C_B,\kappa_B>0$ such that, for $y^1,y^2\in B$ and $z=u^1-u^2$, where $(u^i,\partial_tu^i)=S_ty^i$,
\begin{equation}\label{eq:stab}
\|S_ty^1-S_ty^2\|_{\cH}^2\le C_Be^{-\kappa_Bt}\|y^1-y^2\|^2_{\cH}+C_B\sup_{0\le\tau\le t}\|z(\tau)\|_{L^r}^2,\qquad t\ge0,
\end{equation}
with $r$ given by \eqref{eq:r}. Moreover $\|S_ty^1-S_ty^2\|_{\cH}\le e^{C_Bt}\|y^1-y^2\|_{\cH}$.
\end{lemma}

\begin{proof}
Let $E_z=\frac12(\|\partial_tz\|_\varrho^2+a(z,z))$, so that $\|S_ty^1-S_ty^2\|_{\cH}^2=2E_z(t)$, and let $G=-(\hat f(u^1)-\hat f(u^2))$. By Lemma~\ref{lem:nemytskii} and the boundedness of $B$,
\[
\|G(t)\|\le C_B\|z(t)\|_{L^r}\le C_BE_z(t)^{1/2}.
\]
The pair $Z=(z,\partial_tz)$ is the mild solution of $\partial_tZ=\cA_0Z+(0,\varrho^{-1}G)$, and $G\in C([0,\infty);L^2)$ because $\hat f$ is continuous from $H^1_0$ into $L^2$ by Lemma~\ref{lem:nemytskii}. Let $T_1>0$. We approximate $Z(0)$ in $\cH$ by elements of $D(\cA_0)$ and $G$ in $L^1(0,T_1;L^2)$ by functions of class $C^1$. By \cite[Corollary~4.2.5]{Pazy83}, the corresponding mild solutions are classical solutions, for which
\[
\frac{\dd}{\dd t}\,\frac12\|Z\|_{\cH}^2=(\cA_0Z,Z)_{\cH}+\bigl((0,\varrho^{-1}G),Z\bigr)_{\cH}=-(a\partial_tz,\partial_tz)+(G,\partial_tz),
\]
and they converge to $Z$ in $C([0,T_1];\cH)$ by the variation of constants formula, because the semigroup $e^{t\cA_0}$ generated by $\cA_0$ is a contraction semigroup. Since $\frac12\|Z\|_{\cH}^2=E_z$ and $T_1$ is arbitrary, we obtain
\begin{equation}\label{eq:energy-z}
E_z(t)+\int_{t'}^t(a\partial_tz,\partial_tz)\,\dd\tau=E_z(t')+\int_{t'}^t(G,\partial_tz)\,\dd\tau,\qquad 0\le t'\le t.
\end{equation}
The Lipschitz estimate follows from \eqref{eq:energy-z} and the Gronwall lemma.

Fix $T>T_*$ and write $D=\int_0^T(a\partial_tz,\partial_tz)\,\dd t$. Let $\tilde z$ be the solution of \eqref{eq:linear} with the same data as $z$ at $t=0$, and $\tilde r=z-\tilde z$. Then $\varrho\,\partial_t^2\tilde r-\Lam\tilde r=G-a\partial_tz$ with zero data, and the argument above, applied with $a=0$ and with the source $G-a\partial_tz$, gives
\[
\sup_{[0,T]}E_{\tilde r}\le C\Bigl(\int_0^T\|G-a\partial_tz\|\,\dd t\Bigr)^2\le CT\int_0^T\bigl(\|G\|^2+\|a\|_{L^\infty}(a\partial_tz,\partial_tz)\bigr)\dd t .
\]
By Theorem~\ref{thm:obs} and (H1),
\[
E_z(0)=E_{\tilde z}(0)\le\frac{C_T}{a_0}\int_0^T(a\partial_t\tilde z,\partial_t\tilde z)\,\dd t\le\frac{2C_T}{a_0}\Bigl(D+C\|a\|_{L^\infty}T\sup_{[0,T]}E_{\tilde r}\Bigr)\le C_1\Bigl(D+\int_0^T\|G\|^2\dd t\Bigr),
\]
where we may assume $C_1>1$. By \eqref{eq:energy-z},
\[
D=E_z(0)-E_z(T)+\int_0^T(G,\partial_tz)\,\dd t,
\]
and therefore
\[
E_z(T)\le\Bigl(1-\frac1{C_1}\Bigr)E_z(0)+\int_0^T|(G,\partial_tz)|\,\dd t+\int_0^T\|G\|^2\dd t .
\]
Since
\[
\int_0^T|(G,\partial_tz)|\,\dd t\le\varepsilon\int_0^T\|\partial_tz\|^2\dd t+\frac1{4\varepsilon}\int_0^T\|G\|^2\dd t,\qquad \|\partial_tz(t)\|^2\le Ce^{2C_Bt}E_z(0),
\]
a suitable choice of $\varepsilon$ yields
\[
E_z(T)\le qE_z(0)+C_2\sup_{[0,T]}\|z\|_{L^r}^2,\qquad q=1-\frac1{2C_1}\in(0,1).
\]
Since $B$ is forward invariant, we may apply this estimate on the intervals $[kT,(k+1)T]$. By induction,
\[
E_z(kT)\le q^kE_z(0)+\frac{C_2}{1-q}\sup_{[0,kT]}\|z\|^2_{L^r},
\]
and the Lipschitz estimate on $[kT,(k+1)T]$ gives \eqref{eq:stab} with $\kappa_B=|\ln q|/T$.
\end{proof}

The second result provides regularity of complete trajectories confined in a bounded forward invariant set, before the existence of an attractor is known. For trajectories lying in the attractor it is \cite[Theorem~3.4.19]{Ch15}; the version below, for bounded forward invariant sets, is what the proof of Theorem~\ref{thm:main} requires, and we include its short proof.

\begin{proposition}\label{prop:reg}
Assume \Lh{}, (H1)--(H4) and \Gh{}. Let $B\subset\cH$ be a bounded forward invariant set. There exists $R_B>0$ such that every complete trajectory $\{(u(t),\partial_tu(t)):t\in\R\}\subset B$ satisfies $u(t)\in D(A)$ and $\partial_tu(t)\in H^1_0$ for every $t\in\R$, $\partial_t^2u\in L^\infty(\R;L^2)$, and
\begin{equation}\label{eq:reg}
\|u(t)\|_{H^2}+\|\partial_tu(t)\|_{H^1}+\|\partial_t^2u\|_{L^\infty(\R;L^2)}+\sup_{x\in M}|u(t,x)|\le R_B,\qquad t\in\R .
\end{equation}
Moreover, $\varrho\,\partial_t^2u-\Lam u+a\partial_tu+\hat f(u)=h$ in $L^2(M;TM)$ for a.e. $t\in\R$.
\end{proposition}

\begin{proof}
Let $U(t)=(u(t),\partial_tu(t))$ and $\Theta_B=\sup_{y\in B}\|y\|_{\cH}$. Throughout the proof, $C$ denotes constants depending only on $B$.

\emph{Step 1 (difference quotients).} Fix $\eta\in(0,1]$ and let $z_\eta=u(\cdot+\eta)-u$. For $\tau\le t$, the points $U(\tau+\eta)$ and $U(\tau)$ belong to $B$, and \eqref{eq:stab} applied with initial time $\tau$ gives
\[
\|U(t+\eta)-U(t)\|^2_{\cH}\le C_Be^{-\kappa_B(t-\tau)}(2\Theta_B)^2+C_B\sup_{\tau\le s\le t}\|z_\eta(s)\|_{L^r}^2 .
\]
Letting $\tau\to-\infty$, we obtain
\[
\|U(t+\eta)-U(t)\|^2_{\cH}\le C_B\sup_{s\le t}\|z_\eta(s)\|^2_{L^r},\qquad t\in\R .
\]

\emph{Step 2 (Ehrling's inequality).} Since the embedding $H^1_0\subset L^r$ is compact and $L^r\subset L^2$ is continuous and injective, for every $\varepsilon>0$ there exists $C_\varepsilon$ with $\|v\|_{L^r}\le\varepsilon a(v,v)^{1/2}+C_\varepsilon\|v\|$ for $v\in H^1_0$, see \cite[(3.4.20) in the proof of Proposition~3.4.17]{Ch15}. Let
\[
Z_\eta=\sup_{t\in\R}\|U(t+\eta)-U(t)\|^2_{\cH}\le(2\Theta_B)^2 .
\]
By Step~1,
\[
Z_\eta\le2C_B\varepsilon^2Z_\eta+2C_BC_\varepsilon^2\sup_{s\in\R}\|z_\eta(s)\|^2 .
\]
We choose $\varepsilon$ with $2C_B\varepsilon^2=\frac12$. Since
\[
\|z_\eta(s)\|\le\int_s^{s+\eta}\|\partial_tu\|\,\dd\tau\le C\eta \Theta_B,
\]
we find $Z_\eta\le C\eta^2$.

\emph{Step 3 (time regularity).} Fix $t\in\R$. The quotients $\eta^{-1}z_\eta(t)$ are bounded in $H^1_0$, and they converge to $\partial_tu(t)$ in $L^2$ as $\eta\downarrow0$, because $u\in C^1(\R;L^2)$. Hence a subsequence converges weakly in $H^1_0$, its limit is $\partial_tu(t)$, and $a(\partial_tu(t),\partial_tu(t))\le C$. Moreover $\|\partial_tu(t+\eta)-\partial_tu(t)\|_\varrho\le Z_\eta^{1/2}\le C\eta$. Thus the difference quotients $D_\eta=\eta^{-1}(\partial_tu(\cdot+\eta)-\partial_tu)$ are bounded in $L^\infty(\R;L^2)$, which is the dual of $L^1(\R;L^2)$. A subsequence converges weakly-$*$ to some $\xi$ with $\|\xi\|_{L^\infty(\R;L^2)}\le C$, and for every $\varphi\in C^\infty_c(\R;L^2)$,
\[
\int_\R(D_\eta,\varphi)\,\dd t=-\int_\R\Bigl(\partial_tu,\frac{\varphi(t)-\varphi(t-\eta)}{\eta}\Bigr)\dd t\to-\int_\R(\partial_tu,\varphi')\,\dd t .
\]
Hence $\partial_t^2u=\xi\in L^\infty(\R;L^2)$ in the sense of distributions. Moreover, $Z_\eta\le C\eta^2$ for $\eta\in(0,1]$, together with the triangle inequality for larger increments, shows that $U$ is Lipschitz continuous with values in $\cH$, with a Lipschitz constant depending only on $B$.

\emph{Step 4 (space regularity and the equation).} Let $G(t)=\mathcal F(U(t))$. Since $\mathcal F$ is Lipschitz on bounded sets, $G$ is continuous and bounded by a constant depending only on $B$. Fix $t_0<t$. For $\tau\ge t_0$ the mild solution satisfies
\[
U(\tau)=e^{(\tau-t_0)\cA_0}U(t_0)+\int_{t_0}^\tau e^{(\tau-\sigma)\cA_0}G(\sigma)\,\dd\sigma ,
\]
and Fubini's theorem yields
\[
\int_{t_0}^tU(\tau)\,\dd\tau=\int_0^{t-t_0}e^{\tau\cA_0}U(t_0)\,\dd\tau+\int_{t_0}^t\Bigl(\int_0^{t-\sigma}e^{\tau\cA_0}G(\sigma)\,\dd\tau\Bigr)\dd\sigma .
\]
By \cite[Theorem~1.2.4(b)]{Pazy83}, $\int_0^\theta e^{\tau\cA_0}y\,\dd\tau\in D(\cA_0)$ and $\cA_0\int_0^\theta e^{\tau\cA_0}y\,\dd\tau=e^{\theta\cA_0}y-y$ for $y\in\cH$ and $\theta\ge0$. Hence the integrand of the outer integral in $\sigma$ is continuous with values in $D(\cA_0)$, endowed with the graph norm, and its image under $\cA_0$ is the continuous function $e^{(t-\sigma)\cA_0}G(\sigma)-G(\sigma)$. Since $\cA_0$ is closed, the theorem of Hille \cite[Appendix~C, Proposition~C.4]{EN00} gives
\begin{equation}\label{eq:mild-int}
\int_{t_0}^tU(\tau)\,\dd\tau\in D(\cA_0),\qquad \cA_0\int_{t_0}^tU(\tau)\,\dd\tau=U(t)-U(t_0)-\int_{t_0}^tG(\tau)\,\dd\tau .
\end{equation}
Let
\[
m_\eta(t_0)=\frac1\eta\int_{t_0}^{t_0+\eta}U(\tau)\,\dd\tau,\qquad \eta\in(0,1].
\]
By Step~3, $m_\eta(t_0)\to U(t_0)$ in $\cH$ as $\eta\downarrow0$, and by \eqref{eq:mild-int} and Step~3, $\|\cA_0m_\eta(t_0)\|_{\cH}\le C$. Since the graph of the closed operator $\cA_0$ is a closed subspace of $\cH\times\cH$, it is weakly closed. Therefore every weak limit point of $\cA_0m_\eta(t_0)$ as $\eta\downarrow0$ is equal to $\cA_0U(t_0)$, so that $U(t_0)\in D(\cA_0)$, $\|\cA_0U(t_0)\|_{\cH}\le C$, and the whole family $\cA_0m_\eta(t_0)$ converges weakly to $\cA_0U(t_0)$. By the definition of $\cA_0$, $u(t_0)\in D(A)$ and $\|Au(t_0)\|\le C$, and Lemma~\ref{lem:operatorA} together with the embedding $H^2(M)\subset C(M)$, valid because $n=3$, yields \eqref{eq:reg}.

Finally, by \eqref{eq:mild-int} the functions $t_0\mapsto\cA_0m_{1/k}(t_0)$, $k\in\mathbb N$, are continuous, and they converge weakly to $\cA_0U(t_0)$ for every $t_0$. Hence $t\mapsto\cA_0U(t)$ is weakly measurable and bounded, and, since $\cH$ is separable, it is strongly measurable by the Pettis theorem. Applying \cite[Appendix~C, Proposition~C.4]{EN00} once more to the left-hand side of \eqref{eq:mild-int}, we obtain
\[
U(t)-U(t_0)=\int_{t_0}^t\bigl(\cA_0U(\tau)+G(\tau)\bigr)\dd\tau,
\]
and the Lebesgue differentiation theorem for Bochner integrals shows that $\partial_tU=\cA_0U+G$ for a.e. $t$, whose second component is the stated equation.
\end{proof}

The third result is the rigidity of complete trajectories with constant energy. It replaces the strict Lyapunov property of gradient systems, which we do not prove.

\begin{lemma}\label{lem:rigidity}
Assume \Lh{}, (H1)--(H4) and \Gh{}. Let $B\subset\cH$ be a bounded forward invariant set, and let $\{(u(t),\partial_tu(t)):t\in\R\}\subset B$ be a complete trajectory such that $\Phi(u(t),\partial_tu(t))$ does not depend on $t$. Then $(u(t),\partial_tu(t))=(e,0)$ for all $t$, with $(e,0)\in\cN$.
\end{lemma}

\begin{proof}
By \eqref{eq:energy-identity}, applied on $[s,t]$ for arbitrary $s\le t$, we have $\int_s^t\int_Ma|\partial_tu|^2\,\dd v\,\dd\tau=0$. Hence $a\,\partial_tu=0$ a.e. in $\R\times M$ and, by (H1), $\partial_tu=0$ a.e. in $\R\times\omega$.

Fix $\eta\ne0$ and let $z_\eta=u(\cdot+\eta)-u$. By Proposition~\ref{prop:reg}, $z_\eta\in L^\infty(\R;H^1_0)$, $\partial_tz_\eta\in L^\infty(\R;L^2)$, $\partial_t^2z_\eta\in L^\infty(\R;L^2)$ in the sense of distributions, and, for a.e. $t$,
\[
\varrho\,\partial_t^2z_\eta-\Lam z_\eta+a\,\partial_tz_\eta+\hat f(u(\cdot+\eta))-\hat f(u)=0 .
\]
Here $a\,\partial_tz_\eta=0$, and $\hat f(u(\cdot+\eta))-\hat f(u)=V_\eta z_\eta$, where
\[
V_\eta(t,x)=\int_0^1\partial_uf\bigl(x,u(t,x)+\sigma z_\eta(t,x)\bigr)\,\dd\sigma .
\]
Since the identity above holds in $L^2$ for a.e. $t$, the equation of Theorem~\ref{thm:ucp} with $V=V_\eta$ holds in $\mathcal D'(\R\times\inte M)$. By \eqref{eq:reg}, $|u(t,x)|\le R_B$ and $|u(t+\eta,x)|\le R_B$ for all $t$ and $x$, and $u$ is measurable on $\R\times M$ because $u\in C(\R;L^2)$. Since $\partial_uf$ is continuous on $TM$, $V_\eta$ is a Carath\'eodory composition, hence measurable, and it is bounded by the maximum of $\|\partial_uf\|$ on the compact set $\{(x,\xi)\in TM:|\xi|\le R_B\}$. Moreover, $z_\eta(t)=\int_t^{t+\eta}\partial_tu\,\dd\tau$ in $L^2$, and $\partial_tu(\tau)=0$ a.e. in $\omega$ for a.e. $\tau$, so that $z_\eta=0$ a.e. in $\R\times\omega$. By Theorem~\ref{thm:ucp}, applied on every interval $(t_0,t_0+T)$ with $T>T_*$, we obtain $z_\eta=0$. Since $\eta$ is arbitrary, $u(t)=e$ does not depend on $t$, $\partial_tu=0$, and $(e,0)$ is a stationary point.
\end{proof}

\begin{proof}[Proof of Theorem~\ref{thm:main}]
\emph{Asymptotic smoothness.} We apply the framework of \cite[Section~3.4.3]{Ch15} with $X=H^1_0$, endowed with the norm $a(\cdot,\cdot)^{1/2}$, $Y=L^2_\varrho$ and trivial third component, so that the structural Assumption~3.4.14 of \cite{Ch15} holds. Let $B$ be a bounded forward invariant set. By Lemma~\ref{lem:stab}, the system is asymptotically quasi-stable on $B$ in the sense of \cite[Definition~3.4.15]{Ch15}, with the compact seminorm $\|\cdot\|_{L^r}$ on $H^1_0$ (recall that $r<6$), and with the functions $e^{2C_Bt}$, $C_Be^{-\kappa_Bt}$ and $C_B$ in the roles of the three scalar functions of that definition. By \cite[Proposition~3.4.17]{Ch15}, it is quasi-stable on $B$, and by \cite[Proposition~3.4.3]{Ch15}, $(\cH,S_t)$ is asymptotically smooth.

\emph{Existence and (i).} By Lemma~\ref{lem:Phi-equilibria}, $\Phi$ is a Lyapunov function on $\cH$ in the sense of \cite[Definition~2.4.1]{Ch15}, bounded from above on bounded sets and with bounded sublevel sets. Let $\mathfrak B$ be the set of points $y\in\cH$ through which there passes a complete trajectory along which $\Phi$ is constant. If $y\in\mathfrak B$, the trajectory lies in the bounded forward invariant set $\{\Phi\le\Phi(y)\}$, and Lemma~\ref{lem:rigidity} shows that $y\in\cN$. Conversely $\cN\subset\mathfrak B$. Thus $\mathfrak B=\cN$, which is bounded by Lemma~\ref{lem:Phi-equilibria}. By \cite[Theorem~2.4.21]{Ch15}, $(\cH,S_t)$ possesses a compact global attractor $\mathfrak A=\mathbb M^u(\mathfrak B)=\mathbb M^u(\cN)$.

\emph{(ii) and (iv).} Let $O$ be the open $1$-neighbourhood of $\mathfrak A$ in $\cH$. It is bounded, so that $O\subset\{\Phi\le R_0\}$ for some $R_0$ by Lemma~\ref{lem:Phi-equilibria}, and the closure $B_0$ of $\bigcup_{t\ge0}S_tO$ is a bounded and closed subset of the forward invariant set $\{\Phi\le R_0\}$; it is forward invariant because $S_t$ is continuous and maps $\bigcup_{\tau\ge0}S_\tau O$ into itself. It is absorbing, because $\mathfrak A$ attracts bounded sets and $O$ is a neighbourhood of $\mathfrak A$, so that every bounded set is mapped into $O\subset B_0$ after a finite time. Hence $(\cH,S_t)$ is dissipative and asymptotically quasi-stable on $B_0$, and \cite[Theorem~3.4.18]{Ch15} shows that $\mathfrak A$ has finite fractal dimension. For $y\in B_0$, \eqref{eq:mild-int} holds in the space $\widetilde\cH=L^2\times H^{-1}$, where $\cA_0$ extends to a bounded operator from $\cH$ into $\widetilde\cH$. Since $S_ty$ remains in $B_0$, we obtain $\|S_{t_1}y-S_{t_2}y\|_{\widetilde\cH}\le C|t_1-t_2|$ for $t_1,t_2\in[0,T]$, and the second part of \cite[Theorem~3.4.18]{Ch15} gives (iv).

\emph{(iii).} Every full trajectory in $\mathfrak A$ is contained in $B_0$, and Proposition~\ref{prop:reg} gives (iii) with $R=R_{B_0}$.

\emph{(v).} Let $y\in\cH$. Its positive orbit is contained in $\{\Phi\le\Phi(y)\}$, which is bounded. By \cite[Lemma~3.2.1]{Hale88}, the $\omega$-limit set $\omega(y)$ is nonempty, compact, connected and invariant, and it attracts $y$. Since $\Phi(S_ty)$ is nonincreasing and bounded from below, $\Phi$ is constant on $\omega(y)$. Through every point of the compact invariant set $\omega(y)$ there passes a complete trajectory contained in $\omega(y)$, along which $\Phi$ is constant, so that $\omega(y)\subset\mathfrak B=\cN$. If $\cN$ is finite, the connectedness of $\omega(y)$ forces it to be a single equilibrium, and $S_ty$ converges to it.
\end{proof}

\begin{remark}\label{rem:critical}
Let $\gamma=2$ in (H3), which is the critical growth in dimension three. Proposition~\ref{prop:wellposed} and Lemma~\ref{lem:Phi-equilibria} would remain valid, because $\hat f$ would still be locally Lipschitz from $H^1_0$ into $L^2$, and Lemma~\ref{lem:stab} would remain valid, with $r$ replaced by $2$, on every compact forward invariant set, by the uniform $6$-integrability of its first components. What would be missing is asymptotic smoothness on bounded sets, which for the scalar wave equation is obtained in \cite{CLT08} from a multiplier identity with compensated compactness and in \cite{CMMS26} from Strichartz estimates. For the Lam\'e system under \Gh{}, we have not been able to deduce from \Gh{} the positivity of the interior quadratic form produced by the multiplier $\nabla_{\nabla d}u$, Strichartz estimates up to a Dirichlet boundary are not available for the Lam\'e system, and the Carleman estimates of Section~\ref{sec:HC} do not control the gradient by an $L^2$ source, see Remark~\ref{rem:loss}. The critical case therefore remains open.
\end{remark}

\appendix

%=====================================================================
\section{Rays, geometric control and the obstacle condition}\label{app:rays}
%=====================================================================

In this appendix we relate \Gh{} to the geometric conditions used for the stabilization of waves. In the interior of $M$, the principal symbol of $\varrho\,\partial_t^2-\Lam$ acts on the pressure and shear polarizations as $-\varrho(\tau^2-c_P^2|\xi|^2)$ and $-\varrho(\tau^2-c_S^2|\xi|^2)$, and the projections of the corresponding null bicharacteristics are the geodesics of $g_P$ and of $g_S$, travelled with unit speed with respect to $g_P$ and $g_S$, respectively. We call them pressure rays and shear rays.

\begin{proposition}\label{prop:rays}
Assume \Gh{}, and let $\rho$ be given by Lemma~\ref{lem:U}. For $c\in\{c_S,c_P\}$ let $L_c=\rho^{-1}\max_{\overline U}|\nabla_{g_c}d|_{g_c}$. If $\varsigma:[0,L]\to M$ is a geodesic of $g_c$ with $|\varsigma'|_{g_c}=1$ and $\varsigma([0,L])\cap\omega=\emptyset$, then $L\le\min\{L_c,T_*\}$, where $T_*$ is the geometric time \eqref{eq:Tstar}. Consequently, every shear ray and every pressure ray meets $\omega$ within time $\min\{L_{c_S},T_*\}$ and $\min\{L_{c_P},T_*\}$, respectively, before it can reach $\partial M$.
\end{proposition}

\begin{proof}
Since $\varsigma([0,L])\subset K\subset\overline U$, the function $\vartheta=d\circ\varsigma$ satisfies $\vartheta''=\nabla^2_{g_c}d(\varsigma',\varsigma')\ge2\rho$ by \eqref{eq:convex}, because $\varsigma$ is a geodesic of $g_c$. Moreover $|\vartheta'|=|\langle\nabla_{g_c}d,\varsigma'\rangle_{g_c}|\le\rho L_c$. Hence
\[
2\rho L\le\vartheta'(L)-\vartheta'(0)\le2\rho L_c .
\]
For the bound by $T_*$, let $s_0\in[0,L]$ be a minimum point of $\vartheta$. If $s_0\in(0,L)$, then $\vartheta'(s_0)=0$ and the Taylor formula with the lower bound on $\vartheta''$ gives $\vartheta(s)-\vartheta(s_0)\ge\rho(s-s_0)^2$ for $s\in[0,L]$, which at the endpoint farther from $s_0$ yields $\max\vartheta-\min\vartheta\ge\rho L^2/4$. If $s_0=0$, then $\vartheta'(0)\ge0$ and $\vartheta(L)-\vartheta(0)\ge\rho L^2$, and the case $s_0=L$ is symmetric. In all cases
\[
\rho L^2\le4\bigl(\max_{\overline U}d-\min_{\overline U}d\bigr),
\]
that is, $L\le T_*$. The last assertion follows because $\partial M\subset\omega$.
\end{proof}

The second bound shows that the geometric time $T_*$, which governs Theorems~\ref{thm:ucp} and~\ref{thm:obs}, dominates the time that a shear or a pressure ray can spend in $K$ without meeting $\omega$. Whether $T_*$ can be replaced by this escape time, as is the case for the scalar wave equation under the geometric control condition \cite{BLR92}, is not known to us.

By Proposition~\ref{prop:rays}, every ray reaching $\partial M$ has already met $\omega$, so that the reflections, the gliding rays and the conversions between pressure and shear modes at the boundary described in \cite{BLe01} play no role. In this sense, \Gh{} plays for both families of rays the role of the geometric control condition of \cite{BLR92}. The comparison is made at the level of rays only; we do not invoke, nor claim, a propagation theorem for the Lam\'e system with variable coefficients on a manifold, and the geometric control condition relevant for the elasticity system in Euclidean domains, including the boundary phenomena, is the one described in \cite{BLe01}.

Condition \Gh{} is also related to the escape function conditions of Miller \cite{Mil03} and to the obstacle condition of \cite{CMMS26}. Let $\Gamma\subset\partial M$ be relatively open and $T>0$. Following \cite[Definition~4.1]{Mil03}, the escape potential condition holds for $(\Gamma,T)$ with respect to a metric $\hat g$ if there exists $\varphi\in C^2(M)$ with $|\nabla_{\hat g}\varphi|_{\hat g}\le T/2$, $\nabla^2_{\hat g}\varphi(X,X)\ge|X|^2_{\hat g}$ and $\{x\in\partial M:\partial_\nu\varphi(x)>0\}\subset\Gamma$. By \cite[Proposition~4.2 and Theorem~3.2]{Mil03}, this condition implies that every generalized geodesic of length larger than $T$ leaves $M$ through $\Gamma$. In \cite{CMMS26}, the convexity of the weight is required only on an open set $V$ whose complement is contained in $\omega$, and $\partial M\cap V$ plays the role of an obstacle, on which $\partial_\nu d<0$. Condition \Gh{} is the case $\partial M\cap V=\emptyset$ of this structure, with respect to the two ray metrics $g_S$ and $g_P$. The proof of Proposition~\ref{prop:rays} is the localized version, on the undamped region $K$, of the escape argument of \cite{Mil03}.

%=====================================================================
\section{The Carleman estimate in negative norms}\label{app:IIY}
%=====================================================================

In this appendix we prove the estimate \eqref{eq:IIY} used in the proof of Lemma~\ref{lem:carleman-neg}. The argument is due to Imanuvilov, Isakov and Yamamoto \cite[Section~3]{IIY03}, who proved it for operators with strongly pseudoconvex weights in Euclidean domains, and we claim no originality. We reproduce it for the reader's convenience, because in our setting the strong pseudoconvexity of the weight is replaced by the Riemannian Carleman estimate of Lemma~\ref{lem:carleman-scalar}, and because we do not use the normalization $P(\nabla\phi)=1$ adopted in \cite{IIY03}. A third difference is the following. The estimate of \cite[Theorem~3.2]{IIY03} is stated for $v\in H^2_0(Q)$, whereas in Theorem~\ref{thm:HC0} the unknowns $\delta z$ and $\dd z$ are merely square integrable and the sources are distributions of divergence form, so that the version \eqref{eq:IIY} for $v\in L^2$ is what is needed; it follows from the smooth case by a mollification argument based on Friedrichs' lemma, which we include.

We keep the notation of the proof of Lemma~\ref{lem:carleman-neg}. Thus $x^0=t$ and $x=(x^1,\dots,x^n)$ are the coordinates of the chart $W_\ell''$, the sets $Q$ and $Q^*$ are given by \eqref{eq:QQstar}, and $P_0=\partial_t^2-c^2g^{jl}\partial_j\partial_l$ has smooth coefficients on $\overline{Q^*}$. We fix $\vartheta\in C^\infty_c(\R\times x(W_\ell''))$ with $\vartheta=1$ on a neighbourhood of $\overline{Q^*}$, which is possible because the coefficients of $P_0$ do not depend on $t$ and $\phi$ is defined by \eqref{eq:weight} for all $t\in\R$, and define, on $\R^{n+1}$,
\[
P=\partial_t^2-a^{jl}\partial_j\partial_l,\qquad a^{jl}=\vartheta c^2g^{jl},\qquad \tilde\phi=\vartheta\phi+1-\vartheta ,
\]
so that $P=P_0$ and $\tilde\phi=\phi$ on $Q^*$, and the coefficients $a^{jl}$ and the function $\tilde\phi$ are smooth with bounded derivatives. We write $\phi$ for $\tilde\phi$, $\nabla$ for the full gradient in $(t,x)$, and $\|\cdot\|$ for the norm of $L^2(\R^{n+1})$. The conjugated operator is
\begin{equation}\label{eq:conj}
P_sw=e^{s\phi}P\bigl(e^{-s\phi}w\bigr)=Pw+sR_1w+s^2b\,w,\qquad b=(\partial_t\phi)^2-a^{jl}\partial_j\phi\,\partial_l\phi ,
\end{equation}
where $R_1$ is a first-order operator with smooth bounded coefficients and $b$ is the principal symbol of $P$ evaluated at $\nabla\phi$. Following \cite{IIY03}, we denote by $\Lambda_s$ the Fourier multiplier on $\R^{n+1}$ with symbol $\langle\xi\rangle+s$, where $\langle\xi\rangle=(1+|\xi|^2)^{1/2}$ and $\xi=(\xi_0,\dots,\xi_n)$, and we set $\Lambda_0=\Lambda_s-s$. By Parseval's identity, for $s\ge1$,
\begin{equation}\label{eq:parseval}
\|\Lambda_s^{-1}w\|\le s^{-1}\|w\|,\qquad \|\partial_j\Lambda_s^{-1}w\|\le\|w\|,\qquad \|w\|^2\le2\bigl(s^2\|\Lambda_s^{-1}w\|^2+\|\nabla\Lambda_s^{-1}w\|^2+\|\Lambda_s^{-1}w\|^2\bigr),
\end{equation}
the last inequality because $\langle\xi\rangle+s\le\sqrt2(1+|\xi|^2+s^2)^{1/2}$.

\begin{lemma}\label{lem:B1}
There exist $s_0>0$ and $C>0$ such that, for every $v_0\in H^2(\R^{n+1})$ with compact support in $Q^*$,
\[
s^3\|v_0\|^2+s\|\nabla v_0\|^2\le C\|P_sv_0\|^2,\qquad s\ge s_0 .
\]
\end{lemma}

\begin{proof}
Let $u=e^{-s\phi}v_0$. Pulled back to the manifold, $u$ belongs to $H^2(I_T\times M)$ and has compact support in $I_T\times(W_\ell'\cap U_1)$. On $W_\ell'$ we have $P_0=\partial_t^2-c^2\Delta_g+P_1$ with $P_1=-c^2g^{jl}\Gamma^k_{jl}\partial_k$, and we may multiply the coefficients of $P_1$ by a function of $C^\infty_c(W_\ell'')$ equal to $1$ on $\overline{W_\ell'}$ without changing $P_1u$. Lemma~\ref{lem:carleman-scalar} gives
\[
s\ww{\nabla u}+s\ww{\partial_tu}+s^3\ww{u}\le C\ww{P_0u},\qquad s\ge s_0 .
\]
On $W_\ell'$ the coordinate measure and the coordinate norm of the gradient are comparable with $\dd v$ and $|\nabla\cdot|$, $e^{s\phi}P_0u=P_sv_0$, and $|\nabla v_0|\le e^{s\phi}\bigl(|\nabla u|+s|\nabla\phi|\,|u|\bigr)$. The claim follows.
\end{proof}

\begin{lemma}\label{lem:B2}
There exists $C>0$ such that $\|\Lambda_s^{-1}P_sw-P_s\Lambda_s^{-1}w\|\le C\|w\|$ for every $w\in C^\infty_c(\R^{n+1})$ and every $s\ge1$.
\end{lemma}

\begin{proof}
This is \cite[Lemma~3.3]{IIY03}, with one additional term. For a smooth function $a$ with bounded derivatives, the commutator $[a,\Lambda_0]$ is a pseudodifferential operator of order zero, hence bounded on $L^2(\R^{n+1})$, and
\begin{equation}\label{eq:comm-basic}
[\Lambda_s^{-1},a]=\Lambda_s^{-1}[a,\Lambda_s]\Lambda_s^{-1}=\Lambda_s^{-1}[a,\Lambda_0]\Lambda_s^{-1},\qquad \bigl\|[\Lambda_s^{-1},a]\bigr\|_{L^2\to L^2}\le C(1+s)^{-2}.
\end{equation}
Since $\Lambda_s^{-1}$ commutes with derivatives, by \eqref{eq:conj} it suffices to bound $[\Lambda_s^{-1},a]\partial^\alpha w$ for $|\alpha|\le2$, $s[\Lambda_s^{-1},a]\partial^\beta w$ for $|\beta|\le1$ and $s^2[\Lambda_s^{-1},b]w$. The last term is bounded by $C\|w\|$ by \eqref{eq:comm-basic}. For $|\beta|=1$ we write $[\Lambda_s^{-1},a]\partial_jw=\Lambda_s^{-1}[a,\Lambda_0]\Lambda_s^{-1}\partial_jw$ and use \eqref{eq:parseval}, which gives the bound $Cs^{-1}\|w\|$, and for $|\beta|=0$ we use \eqref{eq:comm-basic}. For $|\alpha|=2$ we write $\partial^\alpha=\partial_j\partial_k$ and $w_1=\Lambda_s^{-1}\partial_kw$, so that $\|w_1\|\le\|w\|$ and, by the identity $[a,\Lambda_0]\partial_j=\partial_j[a,\Lambda_0]-[\partial_ja,\Lambda_0]$,
\[
[\Lambda_s^{-1},a]\partial_j\partial_kw=\Lambda_s^{-1}[a,\Lambda_0]\partial_jw_1=\Lambda_s^{-1}\partial_j[a,\Lambda_0]w_1-\Lambda_s^{-1}[\partial_ja,\Lambda_0]w_1 ,
\]
whose norm is bounded by $C\|w_1\|$ by \eqref{eq:parseval} and the boundedness of the two commutators. The case $|\alpha|\le1$ is contained in the previous ones.
\end{proof}

\begin{lemma}\label{lem:B3}
There exist $s_1>0$ and $C>0$ such that, for every $v\in C^\infty_c(Q)$,
\[
s^3\|\Lambda_s^{-1}v\|^2+s\|v\|^2\le C\|\Lambda_s^{-1}P_sv\|^2,\qquad s\ge s_1 .
\]
\end{lemma}

\begin{proof}
This is \cite[Lemma~3.5]{IIY03}. Fix an open set $Q_1$ with $\overline Q\subset Q_1$ and $\overline{Q_1}\subset Q^*$, let $d_1=\dist(Q,\R^{n+1}\setminus Q_1)>0$, and let $\chi\in C^\infty_c(Q^*)$ with $\chi=1$ on $Q_1$. Let $z=\Lambda_s^{-1}v$, which is a smooth and rapidly decreasing function. Lemma~\ref{lem:B1} applied to $v_0=\chi z$ gives
\[
s^3\|\chi z\|^2+s\|\nabla(\chi z)\|^2\le C\|P_s(\chi z)\|^2 .
\]
By \eqref{eq:conj}, $P_s(\chi z)-\chi P_sz=[P,\chi]z+s[R_1,\chi]z$ is bounded pointwise by $C(|\nabla z|+s|z|)$, so that, by Lemma~\ref{lem:B2} and \eqref{eq:parseval},
\begin{equation}\label{eq:B3-upper}
s^3\|\chi z\|^2+s\|\nabla(\chi z)\|^2\le C\bigl(\|P_sz\|^2+\|\nabla z\|^2+s^2\|z\|^2\bigr)\le C\bigl(\|\Lambda_s^{-1}P_sv\|^2+\|v\|^2\bigr).
\end{equation}
We now bound the left-hand side from below. By the last inequality in \eqref{eq:parseval}, $s\|v\|^2\le C(s^3\|z\|^2+s\|\nabla z\|^2)$. On $Q_1$ we have $\chi z=z$, and on $\R^{n+1}\setminus Q_1$ we use the kernel estimate of \cite[Lemma~3.4]{IIY03}. The Schwartz kernel of $\Lambda_s^{-1}$ is the oscillatory integral
\[
K_s(x,y)=(2\pi)^{-n-1}\int_{\R^{n+1}}e^{i(x-y)\cdot\xi}\bigl(\langle\xi\rangle+s\bigr)^{-1}\dd\xi ,
\]
and $2(n+2)$ integrations by parts in $\xi$, together with the bound $|\partial_\xi^\beta(\xi^\alpha(\langle\xi\rangle+s)^{-1})|\le C_\beta(\langle\xi\rangle+s)^{-2}\langle\xi\rangle^{2-|\beta|}$ for $|\alpha|\le1$ and $|\beta|\ge2$, give
\[
|\partial_x^\alpha K_s(x,y)|\le C(d_1)\,s^{-2}|x-y|^{-2n-4},\qquad |\alpha|\le1,\ |x-y|\ge d_1 .
\]
Since $Q$ is bounded, there is $C_Q$ such that $|x-y|\ge C_Q^{-1}(1+|x|)$ for $y\in Q$ and $x\notin Q_1$, and therefore $|\partial^\alpha z(x)|\le Cs^{-2}(1+|x|)^{-2n-4}\|v\|$ for such $x$ and $|\alpha|\le1$. Consequently
\[
\int_{\R^{n+1}\setminus Q_1}\bigl(s^3|z|^2+s|\nabla z|^2\bigr)\dd x\,\dd t\le Cs^{-1}\|v\|^2 ,
\]
and, combining this with \eqref{eq:B3-upper},
\[
s\|v\|^2+s^3\|\Lambda_s^{-1}v\|^2\le C\bigl(s^3\|\chi z\|^2+s\|\nabla(\chi z)\|^2\bigr)+Cs^{-1}\|v\|^2\le C\|\Lambda_s^{-1}P_sv\|^2+C\|v\|^2+Cs^{-1}\|v\|^2 .
\]
For $s$ large the last two terms are absorbed by $s\|v\|^2$.
\end{proof}

\begin{proof}[Proof of \eqref{eq:IIY}]
This is the argument of \cite[Theorem~3.2]{IIY03}. Assume first that $v\in C^\infty_c(Q)$ and let $w=e^{s\phi}v\in C^\infty_c(Q)$. Since $P=P_0$ on $Q$,
\[
P_sw=e^{s\phi}P_0v=e^{s\phi}h_0+\sum_{j=0}^n\partial_j\bigl(e^{s\phi}g_j\bigr)-s\sum_{j=0}^n(\partial_j\phi)e^{s\phi}g_j .
\]
Lemma~\ref{lem:B3} and \eqref{eq:parseval} give
\[
s\|e^{s\phi}v\|^2\le C\|\Lambda_s^{-1}P_sw\|^2\le C\Bigl(s^{-2}\|e^{s\phi}h_0\|^2+\sum_{j=0}^n\|e^{s\phi}g_j\|^2\Bigr),
\]
where the factor $s$ in the last term of $P_sw$ has been compensated by the norm $s^{-1}$ of $\Lambda_s^{-1}$. This is \eqref{eq:IIY} for smooth $v$. Since the constants depend only on the sets $Q$, $Q_1$ and $Q^*$, the same estimate holds, with constants depending on $Q'$, when $Q$ is replaced by an open set $Q'$ with $\overline Q\subset Q'$ and $\overline{Q'}\subset Q_1$. For a general $v\in L^2$ with compact support in $Q$, let $v_\epsilon=\eta_\epsilon*v$, where $\eta_\epsilon$ is a standard mollifier in $\R^{n+1}$, so that $v_\epsilon\in C^\infty_c(Q')$ for $\epsilon$ small. Writing $P_0=\partial_t^2-\partial_j(a^{jl}\partial_l)+(\partial_ja^{jl})\partial_l$, we have
\[
P_0v_\epsilon=\eta_\epsilon*h_0+\sum_{j=0}^n\partial_j(\eta_\epsilon*g_j)+r^\epsilon+\sum_{j=1}^n\partial_jr_j^\epsilon,\qquad r^\epsilon=\bigl[(\partial_ja^{jl})\partial_l,\eta_\epsilon*\bigr]v,\quad r_j^\epsilon=-\bigl[a^{jl}\partial_l,\eta_\epsilon*\bigr]v ,
\]
and Friedrichs' lemma \cite{Fri44}, in the form of \cite[Lemma~17.1.5]{Hor07}, according to which $[b\,\partial_l,\eta_\epsilon*]v\to0$ in $L^2$ as $\epsilon\to0$ for every $v\in L^2$ and every Lipschitz function $b$, shows that $r^\epsilon\to0$ and $r^\epsilon_j\to0$ in $L^2$. Applying the smooth case to $v_\epsilon$ and letting $\epsilon\to0$ proves \eqref{eq:IIY}.
\end{proof}

\section*{Acknowledgments} This work was supported by the project PID2025-174633NA-I00, funded by \mbox{MICIU/AEI}/\allowbreak 10.13039/\allowbreak 501100011033, and by the Madrid Government (Comunidad de Madrid, Spain) under the Multiannual Agreement 2023--2026 with the Universidad Polit\'ecnica de Madrid, Line A, Emerging PhD Researchers (DOCTORES-EMERGENTES-24-UMDLRU-108-XVLWPC).

\section*{Declaration of competing interest} The author declares that he has no known competing financial interests or personal relationships that could have appeared to influence the work reported in this paper.

\section*{Data availability} No data was used for the research described in the article.


\begin{thebibliography}{99}

\bibitem{AK99} F. Alabau and V. Komornik, Boundary observability, controllability, and stabilization of linear elastodynamic systems, \emph{SIAM J. Control Optim.} \textbf{37} (1999), 521--542.

\bibitem{AM01} G. Alessandrini and A. Morassi, Strong unique continuation for the Lam\'e system of elasticity, \emph{Comm. Partial Differential Equations} \textbf{26} (2001), 1787--1810.

\bibitem{AITY98} D. D. Ang, M. Ikehata, D. D. Trong and M. Yamamoto, Unique continuation for a stationary isotropic Lam\'e system with variable coefficients, \emph{Comm. Partial Differential Equations} \textbf{23} (1998), 599--617.

\bibitem{AC02} M. A. Astaburuaga and R. C. Char\~ao, Stabilization of the total energy for a system of elasticity with localized dissipation, \emph{Differential Integral Equations} \textbf{15} (2002), 1357--1376.

\bibitem{BLR92} C. Bardos, G. Lebeau and J. Rauch, Sharp sufficient conditions for the observation, control, and stabilization of waves from the boundary, \emph{SIAM J. Control Optim.} \textbf{30} (1992), 1024--1065.

\bibitem{BD13} A. Bchatnia and M. Daoulatli, Behavior of the energy for Lam\'e systems in bounded domains with nonlinear damping and external force, \emph{Electron. J. Differential Equations} \textbf{2013} (2013), No. 01, 1--17.

\bibitem{BL02} M. I. Belishev and I. Lasiecka, The dynamical Lam\'e system: regularity of solutions, boundary controllability and boundary data continuation, \emph{ESAIM Control Optim. Calc. Var.} \textbf{8} (2002), 143--167.

\bibitem{Bel01} M. Bellassoued, Unicit\'e et contr\^ole pour le syst\`eme de Lam\'e, \emph{ESAIM Control Optim. Calc. Var.} \textbf{6} (2001), 561--592.

\bibitem{Bel03} M. Bellassoued, Decay of solutions of the elastic wave equation with a localized dissipation, \emph{Ann. Fac. Sci. Toulouse Math. (6)} \textbf{12} (2003), 267--301.

\bibitem{BIY08} M. Bellassoued, O. Imanuvilov and M. Yamamoto, Inverse problem of determining the density and two Lam\'e coefficients by boundary data, \emph{SIAM J. Math. Anal.} \textbf{40} (2008), 238--265.

\bibitem{BY17} M. Bellassoued and M. Yamamoto, \emph{Carleman Estimates and Applications to Inverse Problems for Hyperbolic Systems}, Springer Monographs in Mathematics, Springer, Tokyo, 2017.

\bibitem{BSMS23} L. E. Bocanegra-Rodr\'iguez, M. A. Jorge Silva, T. F. Ma and P. N. Seminario-Huertas, Longtime dynamics of a semilinear Lam\'e system, \emph{J. Dynam. Differential Equations} \textbf{35} (2023), 1435--1456.

\bibitem{BMS26} L. E. Bocanegra-Rodr\'iguez, R. Nunes Monteiro and P. N. Seminario-Huertas, From global to robust exponential attractors for nonlinear elastic waves with variable Lam\'e coefficients, submitted.

\bibitem{BLe01} N. Burq and G. Lebeau, Mesures de d\'efaut de compacit\'e, application au syst\`eme de Lam\'e, \emph{Ann. Sci. \'Ecole Norm. Sup. (4)} \textbf{34} (2001), 817--870.

\bibitem{CDFS10} M. M. Cavalcanti, V. N. Domingos Cavalcanti, R. Fukuoka and J. A. Soriano, Asymptotic stability of the wave equation on compact manifolds and locally distributed damping: a sharp result, \emph{Arch. Ration. Mech. Anal.} \textbf{197} (2010), 925--964.

\bibitem{CMMS26} M. M. Cavalcanti, T. F. Ma, P. Mar\'in-Rubio and P. N. Seminario-Huertas, Supercritical wave equations with $C^1$ nonlinearities: finite-dimensional dynamics and sharp measure-controlled damping, preprint, 2026.

\bibitem{CCD17} C. H. Chan, M. Czubak and M. M. Disconzi, The formulation of the Navier--Stokes equations on Riemannian manifolds, \emph{J. Geom. Phys.} \textbf{121} (2017), 335--346.

\bibitem{CJ02} W. Chen and J. Jost, A Riemannian version of Korn's inequality, \emph{Calc. Var. Partial Differential Equations} \textbf{14} (2002), 517--530.

\bibitem{Ch15} I. Chueshov, \emph{Dynamics of Quasi-Stable Dissipative Systems}, Universitext, Springer, Cham, 2015.

\bibitem{CL08} I. Chueshov and I. Lasiecka, Long-time behavior of second order evolution equations with nonlinear damping, \emph{Mem. Amer. Math. Soc.} \textbf{195} (2008), no. 912.

\bibitem{CL10} I. Chueshov and I. Lasiecka, \emph{Von Karman Evolution Equations. Well-Posedness and Long Time Dynamics}, Springer Monographs in Mathematics, Springer, New York, 2010.

\bibitem{CLT08} I. Chueshov, I. Lasiecka and D. Toundykov, Long-term dynamics of semilinear wave equation with nonlinear localized interior damping and a source term of critical exponent, \emph{Discrete Contin. Dyn. Syst.} \textbf{20} (2008), 459--509.

\bibitem{Ciarlet05} P. G. Ciarlet, \emph{An Introduction to Differential Geometry with Applications to Elasticity}, Springer, Dordrecht, 2005.

\bibitem{DR15} B. Dehman and J.-P. Raymond, Exact controllability for the Lam\'e system, \emph{Math. Control Relat. Fields} \textbf{5} (2015), 743--760.

\bibitem{DRob93} B. Dehman and L. Robbiano, La propri\'et\'e du prolongement unique pour un syst\`eme elliptique. Le syst\`eme de Lam\'e, \emph{J. Math. Pures Appl. (9)} \textbf{72} (1993), 475--492.

\bibitem{EINT02} M. Eller, V. Isakov, G. Nakamura and D. Tataru, Uniqueness and stability in the Cauchy problem for Maxwell and elasticity systems, in \emph{Nonlinear Partial Differential Equations and their Applications. Coll\`ege de France Seminar, Vol. XIV}, Stud. Math. Appl. \textbf{31}, North-Holland, Amsterdam, 2002, 329--349.

\bibitem{EN00} K.-J. Engel and R. Nagel, \emph{One-Parameter Semigroups for Linear Evolution Equations}, Grad. Texts in Math. \textbf{194}, Springer, New York, 2000.

\bibitem{Esc06} L. Escauriaza, Unique continuation for the system of elasticity in the plane, \emph{Proc. Amer. Math. Soc.} \textbf{134} (2006), 2015--2018.

\bibitem{FZ93} E. Feireisl and E. Zuazua, Global attractors for semilinear wave equations with locally distributed nonlinear damping and critical exponent, \emph{Comm. Partial Differential Equations} \textbf{18} (1993), 1539--1555.

\bibitem{Fri44} K. O. Friedrichs, The identity of weak and strong extensions of differential operators, \emph{Trans. Amer. Math. Soc.} \textbf{55} (1944), 132--151.

\bibitem{GLM14} N. Grubi\'c, P. G. LeFloch and C. Mardare, The equations of elastostatics in a Riemannian manifold, \emph{J. Math. Pures Appl.} \textbf{102} (2014), 1121--1163.

\bibitem{Gue00} A. Guesmia, On the decay estimates for elasticity systems with some localized dissipations, \emph{Asymptot. Anal.} \textbf{22} (2000), 1--13.

\bibitem{Hale88} J. K. Hale, \emph{Asymptotic Behavior of Dissipative Systems}, Math. Surveys Monogr. \textbf{25}, Amer. Math. Soc., Providence, RI, 1988.

\bibitem{Hor97} L. H\"ormander, On the uniqueness of the Cauchy problem under partial analyticity assumptions, in \emph{Geometrical Optics and Related Topics} (F. Colombini and N. Lerner, eds.), Progr. Nonlinear Differential Equations Appl. \textbf{32}, Birkh\"auser, Boston, 1997, 179--219.

\bibitem{Hor07} L. H\"ormander, \emph{The Analysis of Linear Partial Differential Operators III. Pseudo-Differential Operators}, Classics in Mathematics, Springer, Berlin, 2007.

\bibitem{Horn98} M. A. Horn, Sharp trace regularity for the solutions of the equations of dynamic elasticity, \emph{J. Math. Systems Estim. Control} \textbf{8} (1998), no. 2, 1--11.

\bibitem{Horn98b} M. A. Horn, Implications of sharp trace regularity results on boundary stabilization of the system of linear elasticity, \emph{J. Math. Anal. Appl.} \textbf{223} (1998), 126--150.

\bibitem{HYY24} M. Hu, X.-G. Yang and J. Yuan, Stability and dynamics for Lam\'e system with degenerate memory and time-varying delay, \emph{Appl. Math. Optim.} \textbf{89} (2024), Paper No. 14.

\bibitem{IIY03} O. Imanuvilov, V. Isakov and M. Yamamoto, An inverse problem for the dynamical Lam\'e system with two sets of boundary data, \emph{Comm. Pure Appl. Math.} \textbf{56} (2003), 1366--1382.

\bibitem{IY05} O. Yu. Imanuvilov and M. Yamamoto, Carleman estimates for the non-stationary Lam\'e system and the application to an inverse problem, \emph{ESAIM Control Optim. Calc. Var.} \textbf{11} (2005), 1--56.

\bibitem{IY05b} O. Yu. Imanuvilov and M. Yamamoto, Carleman estimates for the three-dimensional nonstationary Lam\'e system and application to an inverse problem, in \emph{Control Theory of Partial Differential Equations} (O. Imanuvilov, G. Leugering, R. Triggiani and B.-Y. Zhang, eds.), Lect. Notes Pure Appl. Math. \textbf{242}, Chapman \& Hall/CRC, Boca Raton, 2005, pp. 337--374.

\bibitem{IK08} V. Isakov and N. Kim, Carleman estimates with second large parameter and applications to elasticity with residual stress, \emph{Appl. Math. (Warsaw)} \textbf{35} (2008), 447--465.

\bibitem{IK09} V. Isakov and N. Kim, Carleman estimates with second large parameter for second order operators, in \emph{Sobolev Spaces in Mathematics III. Applications in Mathematical Physics} (V. Isakov, ed.), Int. Math. Ser. (N. Y.) \textbf{10}, Springer, New York, 2009, pp. 135--159.

\bibitem{IWY07} V. Isakov, J.-N. Wang and M. Yamamoto, An inverse problem for a dynamical Lam\'e system with residual stress, \emph{SIAM J. Math. Anal.} \textbf{39} (2007), 1328--1343.

\bibitem{JL13} R. Joly and C. Laurent, Stabilization for the semilinear wave equation with geometric control condition, \emph{Anal. PDE} \textbf{6} (2013), 1089--1119.

\bibitem{Lag83} J. Lagnese, Boundary stabilization of linear elastodynamic systems, \emph{SIAM J. Control Optim.} \textbf{21} (1983), 968--984.

\bibitem{LL19} C. Laurent and M. L\'eautaud, Quantitative unique continuation for operators with partially analytic coefficients. Application to approximate control for waves, \emph{J. Eur. Math. Soc.} \textbf{21} (2019), 957--1069.

\bibitem{LNW10} C.-L. Lin, G. Nakamura and J.-N. Wang, Optimal three-ball inequalities and quantitative uniqueness for the Lam\'e system with Lipschitz coefficients, \emph{Duke Math. J.} \textbf{155} (2010), 189--204.

\bibitem{LM72} J.-L. Lions and E. Magenes, \emph{Non-Homogeneous Boundary Value Problems and Applications. Vol. I}, Grundlehren Math. Wiss. \textbf{181}, Springer, Berlin, 1972.

\bibitem{MMS21} T. F. Ma, J. G. Mesquita and P. N. Seminario-Huertas, Smooth dynamics of weakly damped Lam\'e systems with delay, \emph{SIAM J. Math. Anal.} \textbf{53} (2021), 3759--3771.

\bibitem{MMS25} T. F. Ma, R. N. Monteiro and P. N. Seminario-Huertas, Attractors for locally damped Bresse systems and a unique continuation property, in \emph{Functional Differential Equations and Dynamic Equations on Time Scales} (P. Benevieri and J. G. Mesquita, eds.), Adv. Mech. Math. \textbf{52}, Birkh\"auser, Cham, 2025, pp. 165--196.

\bibitem{MH83} J. E. Marsden and T. J. R. Hughes, \emph{Mathematical Foundations of Elasticity}, Prentice-Hall, Englewood Cliffs, 1983; reprinted by Dover, New York, 1994.

\bibitem{MNY21} Y. Meng, Z.-H. Ning and F. Yang, Internal observability, controllability and stabilization of the inhomogeneous and anisotropic elastic wave equation, \emph{Appl. Math. Optim.} \textbf{84} (2021), 1605--1621.

\bibitem{Mil03} L. Miller, Escape function conditions for the observation, control, and stabilization of the wave equation, \emph{SIAM J. Control Optim.} \textbf{41} (2003), 1554--1566.

\bibitem{Pazy83} A. Pazy, \emph{Semigroups of Linear Operators and Applications to Partial Differential Equations}, Appl. Math. Sci. \textbf{44}, Springer, New York, 1983.

\bibitem{Pet16} P. Petersen, \emph{Riemannian Geometry}, 3rd ed., Grad. Texts in Math. \textbf{171}, Springer, Cham, 2016.

\bibitem{Puj03} J. Pujol, \emph{Elastic Wave Propagation and Generation in Seismology}, Cambridge University Press, Cambridge, 2003.

\bibitem{QW24} Y. Qin and H. Wang, Uniform attractor of a non-autonomous Lam\'e thermoelastic system, preprint, \allowbreak arXiv:2412.09992, 2024.

\bibitem{RZ98} L. Robbiano and C. Zuily, Uniqueness in the Cauchy problem for operators with partially holomorphic coefficients, \emph{Invent. Math.} \textbf{131} (1998), 493--539.

\bibitem{Ruiz92} A. Ruiz, Unique continuation for weak solutions of the wave equation plus a potential, \emph{J. Math. Pures Appl.} \textbf{71} (1992), 455--467.

\bibitem{TL22} X. Tan and G. Liu, Determining Lam\'e coefficients by elastic Dirichlet-to-Neumann map on a Riemannian manifold, preprint, arXiv:2211.06650, 2022.

\bibitem{Tat95} D. Tataru, Unique continuation for solutions to PDE's; between H\"ormander's theorem and Holmgren's theorem, \emph{Comm. Partial Differential Equations} \textbf{20} (1995), 855--884.

\bibitem{Tay23} M. E. Taylor, \emph{Partial Differential Equations I. Basic Theory}, 3rd ed., Appl. Math. Sci. \textbf{115}, Springer, Cham, 2023.

\bibitem{Wec01} N. Weck, Unique continuation for systems with Lam\'e principal part, \emph{Math. Methods Appl. Sci.} \textbf{24} (2001), 595--605.

\bibitem{Yu11} H. Yu, Strong unique continuation for the Lam\'e system with Lipschitz coefficients in three dimensions, \emph{ESAIM Control Optim. Calc. Var.} \textbf{17} (2011), 761--770.

\bibitem{YZSZ26} J. Yu, F. Zhou, Q. Sun and K. Zhu, Attractors for Lam\'e systems with supercritical nonlinearities and damping, \emph{J. Math. Phys.} \textbf{67} (2026), 072705.

\bibitem{Zua90} E. Zuazua, Exponential decay for the semilinear wave equation with locally distributed damping, \emph{Comm. Partial Differential Equations} \textbf{15} (1990), 205--235.

\end{thebibliography}
\end{document}